\documentclass[11pt,a4paper,oneside]{amsart}

\usepackage{amsmath} 
\usepackage{amssymb} 
\usepackage{amsthm} 
\usepackage[english]{babel} 
\usepackage[font=small,justification=centering]{caption} 
\usepackage{enumitem} \setlist{nosep}
\usepackage[T1]{fontenc} 
\usepackage[a4paper]{geometry} 
\usepackage[utf8]{inputenc} 
\usepackage{mathabx} 
\usepackage{mathtools} 
\usepackage[dvipsnames]{xcolor}
\usepackage[pagebackref,breaklinks,linktocpage=true,hidelinks]{hyperref}
\usepackage{tikz-cd} 
\usepackage{xfrac} 
\usepackage{ifthen}
\usepackage{setspace}

\renewcommand*{\backrefalt}[4]{\ifcase #1 (Not cited).\or (Cited p.~#2).\else (Cited pp.~#2).\fi}

\newcounter{claimcount}
\newcounter{enumlabelcount}
\newcounter{thmcount}

\newtheorem{theorem}{Theorem}[section]
\newtheorem{mthm}{Theorem} 
\newtheorem{lemma}[theorem]{Lemma} 
\newtheorem{corollary}[theorem]{Corollary}
\newtheorem{proposition}[theorem]{Proposition}
\theoremstyle{definition}
\newtheorem{definition}[theorem]{Definition}
\newtheorem{remark}[theorem]{Remark}

\newtheorem*{question*}{Question}

\newcommand*{\numberedtheorem}[3]{\theoremstyle{plain}\newtheorem*{makethm\thethmcount}{#1}
    \ifthenelse{\equal{#2}{}}{\begin{makethm\thethmcount}#3\end{makethm\thethmcount}\stepcounter{thmcount}}
    {\begin{makethm\thethmcount}[#2]#3\end{makethm\thethmcount}\stepcounter{thmcount}}}

\makeatletter\newcommand\enumlabel[1][]{\item[#1]
    \refstepcounter{enumlabelcount}\def\@currentlabel{#1}}\makeatother

\makeatletter\def\subsection{\@startsection{subsection}{1}\z@{.7\linespacing\@plus\linespacing}
    {.5\linespacing}{\normalfont\scshape\centering}}\makeatother 

\newenvironment{claim}[1]{\refstepcounter{claimcount}\par\vspace{2mm}\noindent\underline{Claim \theclaimcount:}\hspace{2mm}#1}{}
\newenvironment{claimproof}[1]{\par\vspace{2mm}\noindent\underline{Proof:}\hspace{2mm}#1}
    {\leavevmode\unskip\penalty9999\hbox{}\nobreak\hfill\quad\hbox{$\diamondsuit$}\vspace{2mm}}

\newcommand*{\cal}{\mathcal}
\renewcommand{\hat}{\widehat}
\renewcommand{\bar}{\overline}

\newcommand*{\C}{\mathcal{C}}
\newcommand*{\eps}{\varepsilon}

\newcommand*{\N}{\mathcal{N}}
\newcommand*{\R}{\mathbf{R}}
\newcommand*{\mfp}{\mathfrak p}
\newcommand*{\ppi}{\mfp\pi^i}
\newcommand*{\Q}{\mathcal T}
\newcommand*{\s}{\mathfrak{S}}

\newcommand*{\nest}{\sqsubset}
\newcommand*{\pnest}{\sqsubsetneq}

\newcommand*{\trans}{\pitchfork}

\DeclareMathOperator{\asdim}{asdim}
\DeclareMathOperator{\Aut}{Aut}
\DeclareMathOperator{\cone}{Co}

\DeclareMathOperator{\dist}{\mathsf{d}}
\DeclareMathOperator{\diam}{diam}

\DeclareMathOperator{\hull}{hull}

\DeclareMathOperator{\mcg}{MCG}

\DeclareMathOperator{\rel}{Rel}
\DeclareMathOperator{\SO}{SO}
\DeclareMathOperator{\Sp}{Sp}
\DeclareMathOperator{\stab}{Stab}

\newcommand{\ignore}[2]{\left\{\kern-.7ex\left\{#1\right\}\kern-.7ex\right\}_{#2}}
\newcommand*{\hsp}{\hspace{2mm}}
\newcommand*{\ssp}{\hspace{1mm}}

\tikzset{symbol/.style={draw=none,every to/.append style={edge node={node [sloped, allow upside down, auto=false]{$#1$}}}}}

\begin{document}
\title[Mapping class groups are quasicubical]{Mapping class groups are quasicubical}
\author{Harry Petyt}
\address{Mathematical Institute, University of Oxford, UK}
\email{petyt@maths.ox.ac.uk}

\begin{abstract}
It is proved that the mapping class group of any closed surface with finitely many marked points is quasiisometric to a CAT(0) cube complex. We provide two distinct proofs, one tailored to mapping class groups, and one applying to a larger class of groups.
\end{abstract}

\maketitle
\setcounter{tocdepth}{1}\tableofcontents

\section{Introduction} \label{section:introduction}

In the last few years, there has been a significant amount of work revolving around analogies between mapping class groups and CAT(0) cube complexes. These analogies have taken several forms of varying complexity, including the construction of counterparts of the curve graph \cite{kimkoberda:embedability,hagen:weak,genevois:hyperbolicities}, the development of a cubical version of the Masur--Minsky hierarchy \cite{masurminsky:geometry:2, behrstockhagensisto:hierarchically:1}, and the detection of right-angled Artin subgroups of mapping class groups \cite{koberda:right, clayleiningermangahas:geometry, runnels:effective}.

The effect has been an effusion of new understanding in both settings. For mapping class groups, this has included: confirmation of Farb's quasiflat conjecture \cite{behrstockhagensisto:quasiflats,bowditch:quasiflats}, semihyperbolicity \cite{durhamminskysisto:stable,haettelhodapetyt:coarse}, decision problems for subgroups \cite{bridson:onsubgroups, koberda:right}, and residual properties \cite{dahmanihagensisto:dehn, behrstockhagenmartinsisto:combinatorial}; and on the cubical side: versions of Ivanov's theorem \cite{ivanov:automorphisms, fioravanti:automorphisms}, characterisations of Morse geodesics \cite{abbottbehrstockdurham:largest,incertimedicizalloum:sublinearly}, control on purely loxodromic subgroups \cite{kimkoberda:geometry, koberdamangahastaylor:geometry}, and results on uniform exponential growth \cite{abbottngsprianoguptapetyt:hierarchically}.

Although this viewpoint has been very successful, the two settings are certainly distinct. Indeed, it is well known that (almost all) mapping class groups cannot act properly by semisimple isometries on any complete CAT(0) space \cite[Thm~4.2]{kapovichleeb:actions}; nor can they act properly on (finite- or infinite-dimensional) CAT(0) cube complexes \cite[Thm~1.9]{genevois:median}. In fact, it is unknown whether mapping class groups can act on CAT(0) cube complexes without having a global fixed point, as such an action would imply that they do not have property (T) \cite{nibloreeves:groups}. In any case, the lack of proper actions means that any direct correspondence between mapping class groups and CAT(0) cube complexes that is in some sense ``faithful'' can only be of a purely geometric character.

The main goal of this article is to obtain the strongest direct correspondence one could reasonably hope for. Specifically, we prove the following result.

\begin{mthm} \label{thm:mcgisacubecomplex}
For each closed, oriented surface $S$ with finitely many marked points, there exists a finite-dimensional CAT(0) cube complex $Q$ with a quasimedian quasiisometry \mbox{$\mcg(S)\to Q$}.
\end{mthm}

It is well established that the geometry of a CAT(0) cube complex is completely described by its hyperplane combinatorics \cite{sageev:ends,sageev:codimension}, and this can equally be interpreted in terms of its \emph{median structure} \cite{chepoi:graphs,roller:poc}. Mapping class groups have an analogous  \emph{coarse median structure} \cite{bowditch:coarse} (see Section~\ref{section:preliminaries}), where the role of the median map is played by the ``centroid'' construction of Behrstock--Minsky  \cite{behrstockminsky:centroids}. The fact that the quasiisometry in Theorem~\ref{thm:mcgisacubecomplex} is quasimedian means that, up to a bounded error, it respects these (coarse) median structures, making the correspondence it provides considerably stronger than just a quasiisometry. For instance, the quasimedian property can be used to prove that the \emph{median-quasiconvex} (also known as \emph{hierarchically quasiconvex}) subsets of $\mcg(S)$ exactly correspond to the convex subcomplexes of $Q$; this is made precise in Corollary~\ref{cor:convexity_correspondence}. This has been used to show that the curve graph of $S$ can be coarsely reconstructed from the hyperplane combinatorics of $Q$ \cite[Prop.~7.16]{petytzalloum:constructing}.

One immediate consequence of Theorem~\ref{thm:mcgisacubecomplex} is that mapping class groups admit proper cobounded \emph{quasiactions} on CAT(0) cube complexes, which provides an interesting contrast with the situation for actions. Moreover, the quasiactions produced here are by ``cubical quasi-automorphisms'', not merely by self-quasiisometries. It also follows from Theorem~\ref{thm:mcgisacubecomplex} that $\mcg(S)$ admits a \emph{bounded quasigeodesic bicombing}, and so is \emph{weakly semihyperbolic} in the sense of \cite{alonsobridson:semihyperbolic}. 

A related result to Theorem~\ref{thm:mcgisacubecomplex} has recently been obtained by Hamenst\"adt \cite{hamenstadt:geometry:2a}, who constructs a uniformly locally finite CAT(0) cube complex $C$ with a proper, coarsely onto, Lipschitz map $C\to\mcg(S)$. Hamenst\"adt shows that the space of complete geodesic laminations of $S$ is homeomorphic to the \emph{regular Roller boundary} of $C$. By work of Fern\'os--L\'ecureux--Math\'eus \cite{fernoslecureuxmatheus:contact}, this in turn is homeomorphic to the boundary of a quasitree: the \emph{contact graph} of $C$ \cite{hagen:weak}. However, the map $C\to\mcg(S)$ is not a quasiisometry, and is not known to be quasimedian.

\medskip

This article contains two proofs of Theorem~\ref{thm:mcgisacubecomplex}. The first uses more traditional mapping class group machinery, whereas the second takes place in a setting that is considerably more general than just mapping class groups: the setting of \emph{colourable hierarchically hyperbolic groups} (see Section~\ref{subsection:hhs}). Whilst it should be noted that it is possible to construct hierarchically hyperbolic groups that are not colourable \cite{hagen:non}, all the key examples currently known are colourable \cite{bestvinabrombergfujiwara:constructing, hagenmartinsisto:extra, hughes:lattices, dowdalldurhamleiningersisto:extensions:2, hagenrussellsistospriano:equivariant}. 

\begin{mthm} \label{thm:colourable_HHGs_quasicubical}
Let $G$ be a colourable hierarchically hyperbolic group. There is a finite-dimensional CAT(0) cube complex $Q$ with a quasimedian quasiisometry $G\to Q$.
\end{mthm}

As well as being a generalisation of Theorem~\ref{thm:mcgisacubecomplex}, this can be viewed as a ``globalisation'' of powerful approximation results of Behrstock--Hagen--Sisto and Bowditch \cite{behrstockhagensisto:quasiflats,bowditch:convex}. Namely, those authors show that, under conditions satisfied by all colourable hierarchically hyperbolic groups, the median-quasiconvex \emph{hull} of any finite subset is uniformly (in terms of the number of points) quasimedian quasiisometric to some finite CAT(0) cube complex. This finitary approximation is a key ingredient in both the resolution of Farb's quasiflat conjecture and the recent proofs of semihyperbolicity. In fact, as well as being global, Theorem~\ref{thm:colourable_HHGs_quasicubical} implies a stronger finitary statement where there is no dependence on the number of points; this is a consequence of the correspondence between median-quasiconvexity in $G$ and convexity in $Q$ (Corollary~\ref{cor:convexity_correspondence}). 

A simple consequence of Theorem~\ref{thm:colourable_HHGs_quasicubical} is the recovery of a ``stable cubulation'' result of Durham--Minsky--Sisto for groups \cite[Thm~A]{durhamminskysisto:stable}, albeit with a less tight dimension bound (and without the equivariance of \cite[Thm~4.1]{durhamminskysisto:stable}). See Section~\ref{section:convexity} for more discussion. 

Outside the setting of groups, one can also use the tools of this paper (discussed at the end of the introduction) together with \cite[Thm~4.3, Lem.~4.10]{eskinmasurrafi:large} to prove the following.

\begin{mthm}
Teichm\"uller space, with either the Teichm\"uller metric or the Weil--Petersson metric, admits a quasimedian quasiisometry to a finite-dimensional CAT(0) cube complex.
\end{mthm}

\medskip

An essential part of the proof of Theorem~\ref{thm:colourable_HHGs_quasicubical} is a criterion to determine when a hyperbolic space is quasiisometric to a finite-dimensional CAT(0) cube complex. A result of Haglund--Wise shows that every hyperbolic group is quasiisometric to a locally finite CAT(0) cube complex \cite{haglundwise:combination}, and the argument also applies to uniformly proper hyperbolic spaces. Surprisingly, though, it seems that it the corresponding result for non-proper hyperbolic spaces was not previously known. The following theorem remedies this. (Note that every uniformly proper hyperbolic space has finite asymptotic dimension \cite{gromov:asymptotic,bonkschramm:embeddings}.)

\begin{mthm} \label{mthm:hyp_asdim_ccc}
If $X$ is a hyperbolic space, then $X$ is quasiisometric to a finite-dimensional CAT(0) cube complex if and only if $X$ has finite asymptotic dimension.
\end{mthm}

It should be noted that the fact that the cube complexes in Theorem~\ref{mthm:hyp_asdim_ccc} are only finite-dimensional, not locally finite as in the Haglund--Wise result for hyperbolic groups, is necessary, as Theorem~\ref{mthm:hyp_asdim_ccc} applies to locally infinite trees, for instance. This raises a natural question.

\begin{question*}
Are mapping class groups quasimedian quasiisometric to uniformly locally finite CAT(0) cube complexes? What about colourable hierarchically hyperbolic groups?
\end{question*}

One might hope to use the \emph{Alice's diary} construction of \cite{buyalodranishnikovschroeder:embedding}, which upgrades certain quasiisometric embeddings of hyperbolic spaces in finite products of trees to quasiisometric embeddings in finite products of \emph{binary} trees. However, that construction relies in an essential way on the assumption that the boundary of the hyperbolic space is \emph{doubling}, and unfortunately the doubling condition fails for the relevant spaces in this article.

\medskip

The result that mapping class groups are quasiisometric to finite-dimensional CAT(0) cube complexes (Theorem~\ref{thm:mcgisacubecomplex}) also has interesting implications from the cubical perspective. Indeed, any CAT(0) cube complex that is quasiisometric to a mapping class group must have some peculiar properties.

\begin{mthm} \label{thm:Q_weird}
There exist finite-dimensional CAT(0) cube complexes that have discrete quasiisometry group, are quasiisometric to finitely generated groups, and are not quasiisometric to any CAT(0) cube complex admitting a proper cobounded group action.
\end{mthm}

\begin{proof}
Let $Q$ be a CAT(0) cube complex quasiisometric to the mapping class group of some surface $S$, as given by Theorem~\ref{thm:mcgisacubecomplex}. By quasiisometric rigidity of mapping class groups \cite[Thm~1.1]{behrstockkleinerminskymosher:geometry}, the quasiisometry group of $Q$ is isomorphic to $\mcg(S)$, which is discrete. If a cube complex quasiisometric to $Q$ admitted a proper cobounded group action, then $\mcg(S)$ would be virtually cubulated, contradicting \cite{kapovichleeb:actions}.
\end{proof}

To the best of my knowledge, these are the first examples of CAT(0) cube complexes with these properties. We now discuss examples with subsets of these properties that arise from settings other than mapping class groups.

For examples of CAT(0) cube complexes that are quasiisometric to groups but don't admit proper cobounded group actions, let $\Gamma<\Sp(n,1)$ be a uniform lattice. The group $\Gamma$ is hyperbolic, so is quasiisometric to a CAT(0) cube complex $Q_\Gamma$ by \cite[Thm~1.8]{haglundwise:combination} or Theorem~\ref{mthm:hyp_asdim_ccc}. It also has property~(T) by work of Kazhdan (see \cite[\S3.3]{bekkadelaharpevalette:kazhdans}) and Kostant \cite{kostant:onexistence}. By Pansu's rigidity theorem \cite{pansu:metriques}, if $Q_\Gamma$ admitted a proper cobounded group action, then $\Gamma$ would act with unbounded orbits on some CAT(0) cube complex, contradicting \cite[Thm~B]{nibloreeves:groups}. On the other hand, it can be seen that $Q_\Gamma$ does not have discrete quasiisometry group. Indeed, Schwartz's theorem \cite{schwartz:quasiisometry} implies that the quasiisometry group of $\Gamma$ is isomorphic to the commensurator of $\Gamma$. By Corlette \cite{corlette:archimedean} or Gromov--Schoen \cite{gromovschoen:harmonic}, $\Gamma$ is arithmetic, so Margulis' characterisation of arithmeticity \cite[Thm~9]{margulis:discrete} (also see \cite[\S6.2]{zimmer:ergodic}) implies that the commensurator of $\Gamma$ is Hausdorff-dense in $\Sp(n,1)$. Hence $\Gamma$ has indiscrete quasiisometry group.

For examples with infinite, discrete quasiisometry group, let $\Lambda<\SO(n,1)$ be a nonarithmetic nonuniform lattice, which exists by \cite{gromovpiatetskishapiro:nonarithmetic}. The group $\Lambda$ is hyperbolic relative to virtually abelian subgroups, so by residual finiteness, $\Lambda$ is virtually a colourable hierarchically hyperbolic group \cite[Thm~9.1]{behrstockhagensisto:hierarchically:2}, and hence is quasiisometric to a CAT(0) cube complex $Q_\Lambda$ by Theorem~\ref{thm:colourable_HHGs_quasicubical}. By Margulis' characterisation, $\Lambda$ has finite index in its commensurator, so the quasiisometry group of $Q_\Lambda$ is discrete by Schwartz's theorem. Whether $\Lambda$ can virtually act properly coboundedly on a CAT(0) cube complex is unknown in general, but Wise showed that $\Lambda$ is virtually compact special, hence cocompactly cubulated, when $n=3$ \cite[Thm~17.14]{wise:structure}.

I thank one of the anonymous referees for informing me of the following family of CAT(0) square complexes that have infinite, discrete quasiisometry group and are constructed independently of Theorem~\ref{thm:colourable_HHGs_quasicubical}. Let $T$ be a tree whose automorphism group acts freely and cocompactly (for instance, the universal cover of the below graph). The CAT(0) square complex $T\times T$ also has free cocompact isometry group. Let $\Delta$ be the right-angled octagon complex obtained from $T\times T$ by performing a branched cover of order two at the centre of each square, whose isometry group is also proper and cocompact. The complex $\Delta$ is a Fuchsian building in the sense of \cite{bourdon:surimmeubles}, and hence a theorem of Xie shows that the quasiisometry group of $\Delta$ is isomorphic to $\Aut\Delta$ \cite[Thm~1.2]{xie:quasiisometric}. Subdividing the octagons of $\Delta$ yields a CAT(0) square complex with infinite, discrete quasiisometry group.

\begin{figure}[ht]
\includegraphics[width=2cm]{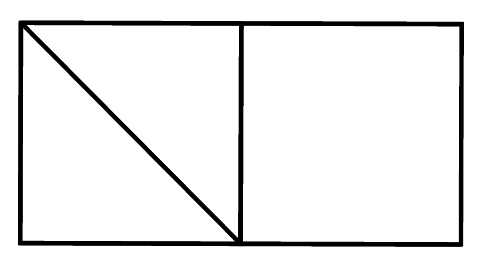}\centering
\end{figure}

\subsubsection*{\emph{\textbf{Outline of the proofs of Theorem~\ref{thm:mcgisacubecomplex}}}} ~

The starting points for the proofs of Theorem~\ref{thm:mcgisacubecomplex} presented in Sections~\ref{section:mainresult} and~\ref{section:buyalo} are constructions of Bestvina--Bromberg--Fujiwara \cite{bestvinabrombergfujiwara:proper} and Buyalo--Dranishnikov--Schroeder \cite{buyalodranishnikovschroeder:embedding}, respectively. The techniques involved are rather different: Section~\ref{section:mainresult} is based around closest-point projections to a family of quasigeodesics in curve graphs that were carefully constructed in \cite{bestvinabrombergfujiwara:proper}, whereas Section~\ref{section:buyalo} is chiefly concerned with a construction for embedding the \emph{hyperbolic cone} on a bounded metric space in a finite product of trees \cite{buyalodranishnikovschroeder:embedding}.

The main point that we use is that these both provide routes to quasiisometrically embedding mapping class groups in finite products of trees. On its own this is not enough to obtain a quasiisometry to a CAT(0) cube complex. Indeed, there is no reason why an arbitrary subset of Euclidean space should admit a quasiisometry to a CAT(0) cube complex, and even if a subset \emph{is} abstractly quasiisometric to a CAT(0) cube complex, it may not be possible to see this from the ambient cubical structure: consider the log-spiral in $\R^2$, for instance. In general the relationship between CAT(0) cube complexes and finite products of trees is surprisingly subtle---there are locally finite CAT(0) cube complexes that cannot be isometrically embedded in any finite product of trees, even in dimensions as low as five \cite{chepoihagen:onembeddings}.

The key to resolving this issue is the quasimedian property. More precisely, given a quasiisometric embedding of $X$ in a finite product of trees, if one has the additional information that it is quasimedian, then it turns out to be possible to obtain a quasiisometry from $X$ to a CAT(0) cube complex (see Proposition~\ref{proposition:hagenpetyt2.11}). Our principal strategy is therefore to show that the embeddings we consider are quasimedian. 

\medskip

\textbf{Proof A.} Section~\ref{section:buyalo} begins by applying the Bestvina--Bromberg--Fujiwara construction as in \cite{bestvinabrombergfujiwara:constructing}. Namely, \cite{bestvinabrombergfujiwara:constructing} produces a finite colouring of the subsurfaces of $S$ such that any two of the same colour overlap. In each colour, the curve graphs of the subsurfaces of that colour can be assembled into a \emph{quasitree of metric spaces}. This is a hyperbolic graph that is built by taking the disjoint union of those curve graphs and adding edges between certain pairs of them; roughly, a pair gets an edge when their subsurface projections to all other subsurfaces of that colour almost coincide. It was shown in \cite{bestvinabrombergfujiwara:constructing} that $\mcg(S)$ quasiisometrically embeds in the product of these hyperbolic graphs, and the embedding was shown to be quasimedian in \cite{hagenpetyt:projection}. This is not the end of the story, because the hyperbolic graphs contain isometrically embedded curve graphs, and so are not all quasitrees.

Nevertheless, Hume observes \cite{hume:embedding} that these hyperbolic graphs can themselves be quasiisometrically embedded in finite products of rooted trees by a construction of Buyalo \cite{buyalo:capacity}. We strengthen this in Section~\ref{section:buyalo} by using a variant construction from \cite{buyalodranishnikovschroeder:embedding} to give an exact characterisation of which hyperbolic spaces admit quasimedian quasiisometric embeddings in finite products of trees. This leads to Theorem~\ref{mthm:hyp_asdim_ccc}, and we observe that the above hyperbolic graphs meet the necessary criteria.

Briefly, given a hyperbolic graph $X$, the trees forming the finite product are built from a sequence of covers of the Gromov boundary of $X$. Each cover is by balls, with radii decaying exponentially along the sequence of covers, and the sequence is coherent in the sense that balls from different terms are either disjoint or nested. The levels of the rooted trees correspond to different terms in the sequence. In order to show that the embedding of $X$ in the product of these trees is quasimedian, we use coherence of the sequence and Lemma~\ref{lem:partial_quasimorphism}, a simplified criterion for a map of hyperbolic spaces to be quasimedian.

\medskip

\textbf{Proof B.} The construction from \cite{bestvinabrombergfujiwara:proper} that is used in the proof in Section~\ref{section:mainresult} also begins with a quasitree of metric spaces, but a more refined collection of metric spaces is involved. More specifically, in \cite{bestvinabrombergfujiwara:proper}, Bestvina--Bromberg--Fujiwara show how to ``decompose'' the curve graph $\C S$ into finitely many orbits of quasigeodesics in such a way that a quasitree of metric spaces can be built for each orbit. These quasitrees of quasigeodesics are themselves quasitrees. By repeating this decomposition in each subsurface and using the colouring from \cite{bestvinabrombergfujiwara:constructing}, they are able to obtain a finite collection of quasitrees that is sufficiently rich for a finite index subgroup $H$ of $\mcg(S)$ to act properly on their product \cite{bestvinabrombergfujiwara:proper}. The representation of $\mcg(S)$ induced by this action of $H$ corresponds to a proper action of $\mcg(S)$ on a finite product of quasitrees.

Because we cannot possibly end up with an equivariant quasiisometry to a CAT(0) cube complex, we can just work with the embedding of $\mcg(S)$ coming from approximating it by its finite-index subgroup $H$ and replacing the finitely many quasitrees by trees. The goal of Section~\ref{section:mainresult} is therefore to prove the quasimedian property for the orbit maps of $H$ on the quasitrees of quasigeodesics.

The strategy for this is to show that orbit maps send \emph{hierarchy paths} in $\mcg(S)$ to paths that project to unparametrised quasigeodesics in each quasitree factor. To see why this is sufficient, note that there is a hierarchy-path triangle $\Delta$ in $\mcg(S)$ with all sides passing through the coarse median $m$ (Lemma~\ref{lem:medianonanhp}). Moreover, the median in the product of quasitrees is just the component-wise median. If the images of the edges of $\Delta$ project to unparametrised quasigeodesics in a quasitree, then the Morse lemma implies that the image of $m$ is uniformly close to all three sides of a geodesic triangle therein, and hence to the median in that quasitree. Knowing that this holds for every quasitree factor gives the quasimedian property.


\subsubsection*{\emph{\textbf{Acknowledgements.}}} ~

I am extremely grateful to my PhD supervisor, Mark Hagen, for sharing his insight and enthusiasm, and for carefully reading several drafts of this article, which greatly benefited it. I thank Matt Durham for informing me about \cite{hamenstadt:geometry:2a}, for a remark about \cite{durhamminskysisto:stable}, and for discussions relating to Teichm\"uller space. I thank Jingyin Huang for his generous explanations about lattices and rigidity relating to Theorem~\ref{thm:Q_weird}. I am grateful to Davide Spriano and Abdul Zalloum for numerous comments that improved the exposition considerably. I would also like to thank Jason Behrstock, Montse Casals-Ruiz, Indira Chatterji, and Ilya Kazachkov for interesting conversations related to this work. I thank the referees for their useful comments and careful reading. This work was supported by an EPSRC DTP at the University of Bristol, studentship 2123260.

\section{Preliminaries} \label{section:preliminaries}
\subsection{Hierarchy structure and notation for mapping class groups} \label{subsection:hierarchy}

Let us start by reviewing some of the hierarchy structure of the mapping class group. This viewpoint was originally developed by Masur--Minsky \cite{masurminsky:geometry:1, masurminsky:geometry:2}, and more recently was axiomatised by Behrstock--Hagen--Sisto with their definition of the classes of hierarchically hyperbolic spaces and groups \cite{behrstockhagensisto:hierarchically:1, behrstockhagensisto:hierarchically:2}. 

Let $S=S_{g,p}$ be the orientable surface of genus $g$ with $p$ marked points. If $3g+p\leq 4$, then $S$ is said to be \emph{sporadic}. The (extended) mapping class group of $S$, denoted $\mcg(S)$, is the group of isotopy classes of (not necessarily orientation-preserving) homeomorphisms of $S$. If $S$ is sporadic, then $\mcg(S)$ is hyperbolic, so we restrict attention to the case where $S$ is not sporadic. It is a classical theorem of Dehn that $\mcg(S)$ is finitely generated \cite{dehn:papers} (in fact it is 2--generated if $g\ge3$ \cite{monden:onminimal} or $p\le1$ \cite{korkmaz:generating}). Fix once and for all a finite generating set for $\mcg(S)$, and let $\dist$ be the corresponding word metric.

The \emph{curve graph} of $S$, introduced by Harvey in \cite{harvey:boundary} and denoted $\C S$, is the graph whose vertex set is the set of isotopy classes of essential simple closed curves on $S$, with an edge joining two vertices if the corresponding classes have disjoint representatives. This graph was shown to be unbounded and hyperbolic by Masur--Minsky \cite{masurminsky:geometry:1}, and since then there have been several articles proving that the hyperbolicity constant is independent of the surface \cite{aougab:uniform, bowditch:uniform, clayrafischleimer:uniform, henselprzytyckiwebb:1-slim}. (This is also true for nonorientable surfaces \cite{kuno:uniform}.) By definition, $\mcg(S)$ acts on $\C S$ by graph automorphisms.

Now let $U$ be an essential proper subsurface of $S$. For simplicity, let us assume that $U$ is not sporadic; see \cite{masurminsky:geometry:2} for how to proceed otherwise. Since $U$ is also a surface, it has a curve graph $\C U$, and this appears as a subgraph of $\C S$, but has diameter at most $2$ in $\C S$ because there is a curve in $S$ that is disjoint from $U$. This means that $\C S$ does not see any information that lives only in $U$. The idea of the Masur--Minsky hierarchy is to overcome this by considering the curve graphs of all subsurfaces of $S$, rather than just $\C S$.

Let $\s$ be the set of all isotopy classes of (possibly sporadic) connected, essential, non-pants subsurfaces of $S$. (For technical reasons, disconnected subsurfaces need to be included in the setting of hierarchical hyperbolicity, but for this article they can be ignored.) Given two (isotopy classes of) subsurfaces $U$ and $V$ in $\s$, there are three ways that the geometries of their curve graphs can interact, based on the configuration of the subsurfaces. They can be disjoint; they can overlap, in which case we write $U\trans V$; or one can be entirely contained in the other, and we write $U\nest V$ if $U$ is a subsurface of $V$. 

If $U\pnest V$, then there is an associated bounded set $\rho^U_V\subset \C V$. In the case where $U$ is not sporadic, this is just the subgraph $\C U \subset \C V$. There is also a bounded set $\rho^U_V\subset\C V$ if $U\trans V$; in this case it is the subgraph $\C(U\cap V)$.

More generally, given a simple closed curve $c$ in $S$ that intersects a subsurface $U$, there is a natural way to define a \emph{projection} of the curve to $\C U$. If $U$ is not an annulus (see \cite[\S2.4]{masurminsky:geometry:2} for the case where $U$ is an annulus), then the image of $c$ is a collection of curves obtained by intersecting $c$ with $U$ and ``closing up'' any loose ends with subsegments of the boundary of $U$. This collection forms a bounded diameter subset of $\C U$. Now fix a marking $m$ of $S$ (see \cite{behrstockkleinerminskymosher:geometry} for background on markings). Every subsurface of $S$ meets $m$. We can use this marking to define, for any subsurface $U\in\s$, a map from $\mcg(S)$ to $\C U$ as follows. Given a mapping class $g$, the subsurface $U$ meets at least one curve that makes up $gm$. For each such curve, take the projection to $\C U$, and define $\pi_U(g)\subset\C U$ to be the bounded set obtained by taking the union of these projections. We call the (set-valued) map $\pi_U:\mcg(S)\to\C U$ a \emph{projection map}. Projection maps are coarsely Lipschitz, though we shall not use this fact directly.

Mapping classes permute the subsurfaces of $S$, and every $g\in\mcg(S)$ induces an isometry $g:\C U\to\C gU$ for each subsurface $U$. By construction, the bounded $\rho$--sets and the projection maps satisfy $g\rho^U_V=\rho^{gU}_{gV}$ and $g\pi_U(g')=\pi_{gU}(gg')$.

The projection maps allow one to view curve graphs of subsurfaces of $S$ as providing coordinates for $\mcg(S)$, by associating with $g$ the tuple $(\pi_U(g))_{U\in\s}$. There are restrictions on the values these coordinates can take coming from the relations between subsurfaces. For instance, Behrstock showed in \cite{behrstock:asymptotic} that if $U$ and $V$ overlap then at most one of $\dist_{\C U}(\pi_U(g),\rho^V_U)$ and $\dist_{\C V}(\pi_V(g),\rho^U_V)$ is greater than some fixed constant that depends only on $S$. This is known as the Behrstock inequality; an elementary proof can be found in \cite{mangahas:recipe}. A similar statement holds when $U\pnest V$ (see \cite[Thm~4.4]{behrstock:asymptotic}), but there is no such restriction when $U$ and $V$ are disjoint.

An important aspect of the hierarchy structure of the mapping class group is the fact that any set of coordinates $x=(x_U)_{U\in\s}$ satisfying Behrstock's inequalities is ``realised'' by a point of $\mcg(S)$ \cite[Thm~4.3]{behrstockkleinerminskymosher:geometry}. That is, there is some $g\in\mcg(S)$ such that $(\pi_U(g))_{U\in\s}$ is uniformly close to $x$ in the supremum metric. This is extremely useful, because it allows one to construct points in $\mcg(S)$ by working only in curve graphs, which are hyperbolic. For example, Behrstock--Minsky used this to construct what they called a \emph{centroid} for a triple of mapping classes \cite{behrstockminsky:centroids}. The idea is as follows. For each subsurface $U$, project the triple to $\C U$ and let $\mu_U$ be the coarse centre in that hyperbolic space. This gives coordinates $(\mu_U)_{U\in\s}$, and it turns out that they satisfy Behrstock's inequalities, so there is some mapping class $\mu$ that realises the tuple $(\mu_U)_{U\in\s}$. This point $\mu$ is declared to be the centroid of the triple.

The centroid construction motivated Bowditch to introduce \emph{coarse median spaces} \cite{bowditch:coarse}, which cover many classes of nonpositively curved spaces besides mapping class groups. For this reason, the centroid is now more commonly called the \emph{coarse median}. Coarse median spaces will be reviewed in Section~\ref{subsection:coarsemedian}. 

To summarise, $\mcg(S)$ has the following structure.
\begin{itemize}
\item   Each subsurface $U\in\s$ has an associated uniformly hyperbolic graph $\C U$.
\item   There is a projection map $\pi_U:\mcg(S)\to\C U$ that sends mapping classes to uniformly bounded subsets of $\C U$. 
\item   If $U\trans V$ or if $U\pnest V$, then there is an associated subset $\rho^U_V\subset\C V$, and this subset is uniformly bounded.
\item   There is equivariance of the form $g\rho^U_V=\rho^{gU}_{gV}$ and $g\pi_U(g')=\pi_{gU}(gg')$.
\item   (Behrstock inequality:) There is a uniform bound on $\min\{\dist_{\C U}(\pi_U(g),\rho^V_U),\dist_{\C V}(\pi_V(g),\rho^U_V)\}$ for $U\trans V$ and $g\in\mcg(S)$.
\item   The coarse median of $g_1, g_2, g_3\in\mcg(S)$ is a mapping class $\mu(g_1,g_2,g_3)$ with the property that, for any $U\in\s$, the projection $\pi_U(\mu(g_1,g_2,g_3))$ is uniformly close to the coarse centre of $\pi_U(g_1)$, $\pi_U(g_2)$, and $\pi_U(g_3)$ in the hyperbolic graph $\C U$.
\end{itemize}

\subsubsection*{\emph{\textbf{Hierarchy constant}}} ~

The hierarchy structure of the mapping class group involves various constants that depend only on the topological type of $S$, some of which have been mentioned above. Fix a constant $E=E(S)\geq1$ that is at least as large as all of these. We think of $E$ as actually being part of the hierarchy structure; see \cite[Rem.~1.6]{behrstockhagensisto:hierarchically:2} for a more explicit description of how $E$ is chosen.

\subsubsection*{\emph{\textbf{Paths and product regions}}} ~

There is a collection of natural paths in the mapping class group, called \emph{hierarchy paths}, that interact well with the hierarchy structure. These abstract the key properties of what were originally called \emph{hierarchies} in \cite{masurminsky:geometry:2}. For $D\ge1$, a \emph{$D$--quasigeodesic} is a $(D,D)$--quasiisometrically embedded interval or line.

\begin{definition}[Hierarchy path]
For a constant $D\geq1$, a $D$--hierarchy path is a $D$--quasigeodesic $\gamma$ in $\mcg(S)$ such that for any subsurface $U\nest S$, the projection $\pi_U\gamma$ is an unparametrised $D$--quasigeodesic.
\end{definition}

Masur--Minsky showed that every pair of mapping classes can be joined by a hierarchy path \cite[Thm~4.6]{masurminsky:geometry:2}. The following is a simple consequence of \cite[Prop.~5.6]{russellsprianotran:convexity}.

\begin{lemma} \label{lem:medianonanhp}
There is a constant $D_1$ such that for any $x_1,x_2,x_3\in\mcg(S)$, there are $D_1$--hierarchy paths $\beta_{ij}$ from $x_i$ to $x_j$ that pass through $\mu(x_1,x_2,x_3)$.
\end{lemma}

Given $g,h\in\mcg(S)$, there are some restrictions on what the hierarchy paths from $g$ to $h$ can look like. In particular, they must ``pass through'' subsurfaces in a particular order; see \cite[Prop.~3.6]{clayleiningermangahas:geometry}. Since the projections of a $D$--hierarchy path are $D$--quasigeodesics, this only makes sense for subsurfaces where $g$ and $h$ are far apart.

\begin{definition}[Relevant subsurface, partial ordering] \label{definition:ordering}
Given $g,h\in\mcg(S)$ and a constant $R\geq20$, we say that a subsurface $U$ is $R$--relevant for the pair $(g,h)$ if $\dist_{\C U}(\pi_U(g),\pi_U(h))\geq R$. We write $\rel_R(g,h)$ for the set of $R$--relevant subsurfaces, and give it a strict partial order \cite[Lem.~4.5]{behrstockkleinerminskymosher:geometry} by setting $U<V$ whenever $U\trans V$ and $\dist_{\C V}(\pi_V(g),\rho^U_V)\leq4$.
\end{definition}




\begin{definition}[Standard product region]
For a subsurface $U$ of $S$, the standard product region of $U$, denoted $P_U$, is the set of all $g\in\mcg(S)$ that satisfy $\dist_{\C V}(\pi_V(g),\rho^U_V)\leq E$ for every subsurface $V\in\s$ with either $V\trans U$ or $U\pnest V$.
\end{definition}


\subsection{Coarse medians} \label{subsection:coarsemedian}

Coarse median spaces were introduced by Bowditch in \cite{bowditch:coarse}, and the class includes many examples of interest, such as mapping class groups, hyperbolic spaces, Teichm\"uller space with either of the usual metrics, CAT(0) cube complexes, and  hierarchically hyperbolic spaces \cite{bowditch:coarse, niblowrightzhang:four, bowditch:large:teichmuller, bowditch:large:weil, behrstockhagensisto:hierarchically:2}. These spaces exhibit a weak kind of nonpositive curvature. Indeed, Bowditch gave a particularly elegant proof that they satisfy a quadratic isoperimetric inequality \cite[Prop.~8.2]{bowditch:coarse}. 

Although we shall not use it directly, the definition of a coarse median space has been included for completeness; see \cite{bowditch:coarse,bowditch:notes}. Recall that the median of three vertices in a finite-dimensional CAT(0) cube complex is the unique vertex lying on an $\ell^1$--geodesic between each pair. 

\begin{definition}[Coarse median space]
A metric space $(X,\dist)$ is a coarse median space if there is a map $\mu:X^3\to X$ and a function $h$ such that the following conditions hold.
\begin{itemize}
\item   For any $x,x',y,y',z,z'\in X$ we have
\[  \dist(\mu(x,y,z),\mu(x',y',z')) \le h(1)(1+\dist(x,x')+\dist(y,y')+\dist(z,z')). \] 
\item   For all $n\in\mathbf N$, if $A\subset X$ has cardinality at most $n$, then there is a finite CAT(0) cube complex $Q$ with maps $f:A\to Q$ and $\bar f:Q\to X$ such that
\begin{align*}
&-\hspace{2mm} \dist\big(\bar f\mu(v_1,v_2,v_3),\mu(\bar f(v_1),\bar f(v_2),\bar f(v_3))\big)\le h(n) \text{ for all } v_1,v_2,v_3\in Q; \\
&-\hspace{2mm} \dist(a,\bar ff(a))\le h(n) \text{ for all } a\in A.
\end{align*}
\end{itemize}
\end{definition}

More important for us will be the notion of a \emph{quasimedian map}, which is the natural morphism for the setting of coarse median spaces.

\begin{definition}[Quasimedian]
Let $(X,\dist_X,\mu_X)$ and $(Y,\dist_Y,\mu_Y)$ be coarse median spaces. A map $\phi:X\to Y$ is quasimedian if there is a constant $k$ such that 
\[ \dist_Y\big(\phi\mu_X(x_1,x_2,x_3),\mu_Y(\phi(x_1),\phi(x_2),\phi(x_3))\big) \le k \]
holds for all $x_1,x_2,x_3\in X$. 
\end{definition}

We say that a subset $Y$ of a metric space $X$ is \emph{coarsely connected} if there is a constant~$r$ such that for any $y,y'\in Y$ there is a sequence $y=y_0,y_1,\dots,y_n=y'$ with $\dist(y_{i-1},y_i)\le r$ for all $i\le n$. If we can take $r=1$, then we say that $Y$ is \emph{1--connected}. The proofs of our main results rely in an essential way on the following, which is based on an result due independently to Bowditch \cite[\S4]{bowditch:convex} and Fioravanti \cite[Prop.~4.1]{fioravanti:coarse}. 

\begin{proposition}[{\cite[Prop.~2.12]{hagenpetyt:projection}}] \label{proposition:hagenpetyt2.11}
Any coarsely connected coarse median space $X$ that admits a quasimedian quasiisometric embedding $\Phi$ in a finite-dimensional CAT(0) cube complex $Q$ is quasimedian quasiisometric to a finite-dimensional CAT(0) cube~complex.
\end{proposition}

\begin{proof}
The finite-dimensionality of $Q$ lets us perturb $\Phi$ to map into the 0--skeleton $Q^0\subset Q$, and the image is coarsely connected. Since $\Phi$ is quasimedian, the median of any three points of $\Phi(X)$ lies uniformly close to $\Phi(X)$. According to \cite[Prop.~2.8]{hagenpetyt:projection}, this shows that $\Phi(X)$ is at bounded Hausdorff-distance from a 1--connected median subalgebra $M\subset Q^0$. As mentioned in \cite[\S2]{bowditch:convex}, 1--connected median subalgebras of $Q^0$ are isometrically embedded, so $\Phi$ is a quasiisometry from $X$ to the CAT(0) subcomplex of $Q$ whose 0--skeleton is $M$, again relying on finite-dimensionality.
\end{proof}

Let us now state a few useful facts about coarse medians and hyperbolicity.

\begin{proposition}[{\cite[Thm~4.2]{niblowrightzhang:four}}] \label{prop:hyperbolic_median_unique}
Coarse medians on hyperbolic spaces are unique up to bounded error.
\end{proposition}

\begin{lemma} \label{lem:qieimpliesqm}
If $X$ and $Y$ are hyperbolic spaces, then any map $\phi:X\to Y$ that sends geodesics to uniform unparametrised quasigeodesics is quasimedian. In particular, any quasiisometric embedding of hyperbolic spaces is quasimedian.
\end{lemma}

\begin{proof}
According to Proposition~\ref{prop:hyperbolic_median_unique}, there is no ambiguity in the median operations. Let $x_1,x_2,x_3\in X$, and for each pair $(i,j)$ let $\gamma_{ij}$ be a uniform quasigeodesic from $x_i$ to $x_j$ that passes through the median $m=\mu_X(x_1,x_2,x_3)$. The image $\phi\gamma_{ij}$ is a uniform unparametrised quasigeodesic, so lies at bounded Hausdorff-distance from a geodesic with the same endpoints by the Morse lemma. Thus $\phi(m)$ is uniformly close to $\mu_Y(\phi(x_1),\phi(x_2),\phi(x_3))$.
\end{proof}

Let $X$ and $Y$ be metric spaces equipped with $n$--ary operations $f_X:X^n\to X$ and $f_Y:Y^n\to Y$. We say that a map $\phi:X\to Y$ is a \emph{coarse morphism} with respect to $f_X$ and $f_Y$ if there is a constant $k$ such that $\phi f_X(x_1,\dots,x_n)$ is $k$--close to $f_Y(\phi(x_1),\dots,\phi(x_n))$ for every $(x_1,\dots,x_n)\in X^n$.

\begin{lemma} \label{lem:partial_quasimorphism}
Let $X$ and $Y$ be hyperbolic spaces. If a coarsely Lipschitz map $\phi:X\to Y$ is a coarse morphism with respect to the binary operations $\mu_X(\cdot,\cdot,x_0):X^2\to X$ and $\mu_Y(\cdot,\cdot,\phi(x_0)):Y^2\to Y$ for some $x_0\in X$, then $\phi$ is quasimedian.
\end{lemma}

\begin{proof}
Let $x_1,x_2\in X$, and let $\gamma$ be a geodesic from $x_1$ to $x_2$. Let $\gamma_i$ be a uniform quasigeodesic from $x_0$ to $x_i$ that passes through $m=\mu_X(x_1,x_2,x_0)$. Let $\gamma_i'\subset\gamma_i$ be the subsegment from $m$ to $x_i$. Then $\gamma$ lies at bounded Hausdorff-distance from the uniform quasigeodesic $\gamma_1'\cup\gamma_2'$. 

The coarse morphism property of $\phi$ tells us that the $\gamma_i$ get mapped to uniform unparametrised quasigeodesics, and moreover that $\phi(m)$ is uniformly close to $\mu_Y(\phi(x_1),\phi(x_2),\phi(x_0))$. In particular, the coarse intersection of $\phi\gamma_1'$ with $\phi\gamma_2'$ is uniformly bounded. This shows that $\phi\gamma_1'\cup\phi\gamma_2'$ is a uniform unparametrised quasigeodesic. Since $\gamma$ lies at bounded Hausdorff-distance from $\gamma_1'\cup\gamma_2'$, this implies that $\phi\gamma$ is a uniform unparametrised quasigeodesic.

We have shown that $\phi$ sends geodesics to uniform unparametrised quasigeodesics, so we are done by Lemma~\ref{lem:qieimpliesqm}.
\end{proof}

\subsection{The Bestvina--Bromberg--Fujiwara construction} \label{subsection:bbf}

Another fruitful way of studying mapping class groups is via the \emph{projection complex} techniques introduced by Bestvina--Bromberg--Fujiwara \cite{bestvinabrombergfujiwara:constructing}. These allow one to assemble the curve graphs of the subsurfaces of $S$ together into a finite collection of hyperbolic spaces in such a way that $\mcg(S)$ virtually acts on their product, with orbit maps being quasiisometric embeddings. The construction is considerably more general than this, and has many applications, for example \cite{bartelsbestvina:farrell, eskinmasurrafi:large, dahmani:normal, balasubramanya:acylindrical, claymangahas:hyperbolic}. 

Following \cite[\S4]{bestvinabrombergfujiwara:constructing}, let $\bf Y$ be a collection of geodesic metric spaces with specified subsets $\pi_Y(X)\subset Y$ for any distinct $X,Y\in \bf Y$. Let $\dist^\pi_Y(X,Z)$ denote the quantity $\diam(\pi_Y(X)\cup\pi_Y(Z))$; in particular, if every $\pi_Y(X)$ is a singleton, as will be the case for us in Section~\ref{section:mainresult}, then $\dist^\pi_Y(X,Z)=\dist_Y(\pi_Y(X),\pi_Y(Z))$. We say that $(\bf Y,\{\pi_Y\})$ \emph{satisfies the projection axioms} with constant $\xi$ if the following hold for any distinct $X,Y,Z\in\bf Y$.
\begin{enumerate}
\enumlabel[(P0)] $\diam(\pi_Y(X))\leq\xi.$       \label{projection0} 
\enumlabel[(P1)] If $\dist^\pi_Z(X,Y)>\xi$, then $\dist^\pi_X(Y,Z)\le \xi$. \label{projection1} 
\enumlabel[(P2)] The set $\{W:\dist^\pi_W(X,Y)>\xi\}$ is finite. \label{projection2}
\end{enumerate}
Moreover, if a group $G$ acts on $\bf Y$ and each element $g$ of $G$ induces isometries $\bar g:Y\to gY$, then we say that the projection axioms are satisfied $G$--equivariantly if $\overline{g_1g_2}=\bar g_1\bar g_2$ and $\bar g\pi_Y(X)=\pi_{gY}(gX)$ hold for any distinct $X,Y\in\bf Y$.

As an example, let $\bf Y$ comprise the curve graphs associated to a collection of pairwise overlapping subsurfaces of $S$, with $\pi_Y(X)=\rho^X_Y$. Then \ref{projection1} is the Behrstock inequality, and \ref{projection2} follows from \cite[\S6]{masurminsky:geometry:2} or \cite[Lem.~5.3]{bestvinabrombergfujiwara:constructing}.

\subsubsection*{\emph{\textbf{The quasitree of metric spaces}}} ~

Recall that a \emph{quasitree} is a geodesic metric space that is quasiisometric to a tree. (We could equivalently demand that it be $(1,C)$--quasiisometric to a tree \cite{kerr:tree}.)

Suppose that $(\bf Y,\{\pi_Y\})$ satisfies the projection axioms with constant $\xi$. In order for the constructions to work, we need to perturb the distance functions $\dist^\pi_Y$; see \cite[\S3.2]{bestvinabrombergfujiwara:constructing} for how to do this. (Actually, by \cite[Thm~4.1]{bestvinabrombergfujiwarasisto:acylindrical}, we could instead perturb the sets $\pi_Y(X)$, but perturbing the distance function will have less inertia.) The details will not be important for us, only that the alteration was small: writing $\dist^\flat_Y$ for the modified distance function, there is a constant $\delta$, depending only on $\xi$, such that 
\begin{align}
\dist^\pi_Y(X,Z)-\delta \hspace{2mm}\le\hspace{2mm} \dist^\flat_Y(X,Z) \hspace{2mm}\le\hspace{2mm} \dist^\pi_Y(X,Z) \label{inequality:changedistance}
\end{align}
holds for any three distinct elements of $\bf Y$. 

The main thing is that this modified distance is what is used to construct the \emph{quasitree of metric spaces}. By the results of \cite{bestvinabrombergfujiwara:constructing}, there is a constant $\Theta$ depending only on $\xi$ such that all of what follows holds for any value of $K\geq\Theta$. To construct the quasitree of metric spaces $\C_K\bf Y$, begin with the disjoint union $\bigsqcup_{Y\in\bf Y}Y$, then attach an edge of length $L$ from each point in $\pi_Y(X)$ to each point in $\pi_X(Y)$ whenever $\dist^\flat_Z(\pi_Z(X),\pi_Z(Y))\le K$ for every other $Z\in\bf Y$. The constant $L$ is determined by $K$ and $\xi$. If the projection axioms are satisfied $G$--equivariantly, then the construction of $\C_K\bf Y$ naturally gives an action of $G$ on $\C_K\bf Y$ by isometries.

What makes this construction so useful is the fact that if the spaces that make up $\bf Y$ are uniformly hyperbolic then the quasitree of metric spaces $\C_K\bf Y$ is also hyperbolic \cite[Thm~4.17]{bestvinabrombergfujiwara:constructing}, and if they are uniformly quasiisometric to trees then $\C_K\bf Y$ is also a quasitree \cite[Thm~4.14]{bestvinabrombergfujiwara:constructing}.

The quasitree of metric spaces comes with a ``distance formula''. That is, the distance between two points in $\C_K\bf Y$ can coarsely be measured by considering the projections of the points to the component metric spaces---this is similar to the situation for mapping class groups \cite[Thm~6.12]{masurminsky:geometry:2}. For $X,Y\in\bf Y$ and $x\in X$, set $\pi^\flat_Y(x)=x$ if $Y=X$, and $\pi^\flat_Y(x)=\pi_Y(X)$ otherwise. The quantity $\ignore{a}{b}$ is equal to $a$ if $a\ge b$, and $0$ otherwise.

\begin{theorem}[Distance formula, {\cite[Thm~4.13]{bestvinabrombergfujiwara:constructing}}] \label{thm:bbfdf}
There is a constant $K'>K$ such that, for any $X$ and $Y$ in $\bf Y$, if $x\in X$ and $y\in Y$, then 
\[
\frac{1}{2}\sum_{Z\in\bf Y}\ignore{\dist^\flat_Z(\pi^\flat_Z(x),\pi^\flat_Z(y))}{K'}
    \hspace{2mm}\le\hspace{2mm}\dist_{\C_K\bf Y}(x,y)
    \hspace{2mm}\le\hspace{2mm}6K+4\sum_{Z\in\bf Y}\ignore{\dist^\flat_Z(\pi^\flat_Z(x),\pi^\flat_Z(y))}{K}.
\]
\end{theorem}

The way we shall use this is to say that if $\dist_{\C_K\bf Y}(x,y)>6K$, then there is some $Z\in\bf Y$ that has $\dist^\flat_Z(\pi^\flat_Z(x),\pi^\flat_Z(y))\geq K$. We shall not directly use the lower bound, though we rely on its use in \cite{bestvinabrombergfujiwara:proper}.

\section{Proof using quasitrees of quasigeodesics} \label{section:mainresult}

In this section, we use the tools of \cite{bestvinabrombergfujiwara:proper} to give our first proof of Theorem~\ref{thm:mcgisacubecomplex}. Let us begin by describing some of the technical components of \cite[\S4]{bestvinabrombergfujiwara:proper}. 

Firstly, given subsurfaces $U$ and $V$ of $S$ that are equal or transverse, and given a quasigeodesic $\gamma$ in the curve graph $\C U$, define a map $\mfp_\gamma:\C V\to\gamma$ as follows. For $x\in\C V$, set $\mfp_\gamma(x)$ to be a closest point to $x$ in $\gamma$ if $V=U$, and a closest point to $\rho^V_U$ in $\gamma$ if $U\trans V$. If we have a collection of quasigeodesics that is invariant under the action of a subgroup $H<\mcg(S)$, then these projections can be chosen so that $h\mfp_\gamma(x)=\mfp_{h\gamma}(hx)$ for every $h\in H$. 

By \cite[Prop.~5.8]{bestvinabrombergfujiwara:constructing}, the set $\s$ of subsurfaces admits a finite partition $\s=\bigsqcup_{i=1}^\chi\s_i$ such that if $U,V\in\s_i$, then $U$ and $V$ overlap. Moreover, there is a finite-index subgroup $H^\mathrm{clr}$ of $\mcg(S)$ that preserves this colouring. The constructions of \cite[\S4]{bestvinabrombergfujiwara:proper} provide the following.
\begin{itemize}
\item   A constant $\lambda$ and a finite collection of distinguished $\lambda$--quasigeodesics $A=\{\gamma_1,\dots,\gamma_m\}$, where $\gamma_i$ lies in the curve graph $\C U_i$ of a subsurface $U_i$. We may have $U_i=U_j$.
\item   A finite-index subgroup $H<H^\mathrm{clr}$ with the property that the diameter of $\mfp_{\gamma_i}(h\gamma_i)$ is uniformly bounded (in terms of $E$) whenever $h\in H$ stabilises $U_i$ but not $\gamma_i$. 
\item   $H$ acts on the $\mcg(S)$--orbit of $A$, and the $\gamma_i$ form a transversal of the $H$--orbits. Although the sets $A_i=H\cdot\{\gamma_i\}$ are pairwise disjoint, the $\mcg(S)$--orbits of some of the $\gamma_i$ will coincide.
\item   A constant $\xi=\xi(E,\lambda)$ such that each $(A_i,\{\mfp_{h\gamma_i}(h'\gamma_i)\})$ satisfies the projection axioms $H$--equivariantly with constant $\xi$. 
\end{itemize}
The last point allows the construction of a quasitree $\C_K A_i$ with an isometric action of $H$, as described in Section~\ref{subsection:bbf}, and Bestvina--Bromberg--Fujiwara prove the following.

\begin{theorem}[{\cite[\S4.9]{bestvinabrombergfujiwara:proper}}] \label{thm:bbfqie}
For $K$ sufficiently large in terms of $E$ and $\lambda$, any orbit map $H\to\prod_{i=1}^m\C_K A_i$ is a quasiisometric embedding of $H$ in a product of quasitrees.
\end{theorem} 

In view of Proposition~\ref{proposition:hagenpetyt2.11}, our aim is to show that these orbit maps are quasimedian, where the coarse median on $H$ is inherited from $\mcg(S)$. We start by making a particular choice of orbit map that simplifies a number of our arguments.

\subsubsection*{\textbf{\emph{The choice of orbit.}}} ~

We start by altering the choice of the axis $\gamma_i$ inside its $H$--orbit, which allows for some notational simplifications. Let $h\in P_{U_i}$ lie in the product region of $U_i$, so that $1\in P_{h^{-1}U_i}$. Replacing $\gamma_i$ by $h^{-1}\gamma_i$, we may assume that $1\in P_{U_i}$. By definition of $P_{U_i}$, this ensures that, for any $g\in H$, we have
\begin{align}
\dist_{\C V}(\pi_V(g),\rho^{gU_i}_V)\leq E \hsp\text{ whenever }\hsp V\trans gU_i. \label{eq:nearrho}
\end{align}
Moreover, we can choose $\gamma_i$ inside its $\stab_H(U_i)$--orbit to minimise $\dist_{U_i}(\gamma_i,\pi_{U_i}(1))$.

For each $i$, let $\psi_i:H\to\C_K A_i$ be given by 
\[ \psi_i(h) \hsp=\hsp h\mfp_{\gamma_i}\pi_{U_i}(1) \hsp=\hsp \mfp_{h\gamma_i}\pi_{hU_i}(h). \]
Then $\Psi=(\psi_1,\dots,\psi_m)$ defines an orbit map $H\to\prod_{i=1}^m\C_K A_i$.

Compositions as in the definition of $\psi_i$ will arise frequently, so let us adopt the convention of writing 
\[\ppi_h=\mfp_{h\gamma_i}\pi_{hU_i}:H\to h\gamma_i.\] 
For example, $\psi_i(h)$ is the image of $h$ under the composition of $\ppi_h$ with the inclusion $h\gamma_i\hookrightarrow\C_K A_i$. This also illustrates the fact that $h\gamma_i$ is naturally a subspace of both $\C U_i$ and $\C_KA_i$, although there is no sensible map between these two.

\subsubsection*{\textbf{\emph{Overview of important constants.}}} ~

The hierarchy constant $E$ and the quasigeodesic constant $\lambda$ are the ``independent variables''. As described, these determine the projection-axiom constant $\xi$. From these we obtain $\nu\geq10E$ and $\lambda'\ge\lambda+\xi$, which should also be thought of as small; their roles are mostly technical. See Lemmas~\ref{lem:behrstockaxis} and~\ref{lemma:lambdad}.

For the next ``order of magnitude'', fix $D\geq\max\{10\nu, D_1, \Theta\}$, where $D_1$ is the constant from Lemma~\ref{lem:medianonanhp} and $\Theta$ is as in Section~\ref{subsection:bbf}. Using this, we shall define a larger number $D'\geq \lambda'(D+50\lambda')$ in Lemma~\ref{lemma:lambdad}.

Larger still, fix $K>101D'$ large enough for Theorem~\ref{thm:bbfqie} to apply. Because $K$ is used to construct the quasitree of metric spaces $\C_KA_i$, it is perhaps more useful to think of choosing $K$ first, then choosing $D$ depending on this so that $D$ and $D'$ lie in a particular interval between $\max\{E,\lambda\}$ and $K$, which is nonempty as long as $K$ was chosen to be sufficiently large.

With $K$ now fixed, let us write $\Q_i$ for the quasitree $\C_KA_i$ from now on, both to simplify notation and to avoid confusion with the curve graphs $\C U_i$.

\subsubsection*{\emph{\textbf{Proof of {Theorem~\ref{thm:mcgisacubecomplex}}}}} ~

With this set-up established, we can now prove Theorem~\ref{thm:mcgisacubecomplex}. The following proposition is the technical heart of the argument; we prove it in Section~\ref{subsection:hptoqg}.

\begin{proposition} \label{prop:hptoqg}
There is a constant $q$ such that if $\beta$ is a $D$--hierarchy path in $H$, then $\psi_i\beta$ is an unparametrised $q$--quasigeodesic.
\end{proposition}

It is worth observing that this is not automatic from the fact that hierarchy paths are quasigeodesics and $\Psi$ is a quasiisometric embedding. For example, a log-spiral in the Euclidean plane is a quasigeodesic, but its projections to 1--dimensional subspaces are not. We actually only need the weaker statement that $\psi_i\beta$ is at a uniformly bounded Hausdorff-distance from the geodesic between its endpoints, but Proposition~\ref{prop:hptoqg} is perhaps the most natural way to show this.

Here follows the proof of Theorem~\ref{thm:mcgisacubecomplex}, assuming Proposition~\ref{prop:hptoqg}. 

\begin{proof}[Proof of Theorem~\ref{thm:mcgisacubecomplex}]
It suffices to prove the result for the finite-index subgroup $H<\mcg(S)$, so let us show that $\Psi:H\to\prod_{i=1}^m\Q_i$ is quasimedian. Since the coarse median in $\prod_{i=1}^m\Q_i$ is defined component-wise, this amounts to showing that each $\psi_i$ is quasimedian.

For this, take $x_1,x_2,x_3\in H$ and let $m=\mu(x_1,x_2,x_3)$. According to Lemma~\ref{lem:medianonanhp}, there are $D$--hierarchy paths $\beta_{jk}$ from $x_j$ to $x_k$ that pass through $m$. Applying $\psi_i$ gives a triangle in $\Q_i$ whose vertices are $\psi_i(x_j)$ and whose edges are $\psi_i\beta_{jk}$, and $\psi_i(m)$ lies on all three sides of this triangle. Moreover, Proposition~\ref{prop:hptoqg} tells us that the sides of this triangle are $q$--quasigeodesics. It now follows from hyperbolicity of $\Q_i$ that $\psi_i(m)$ is uniformly close to the median of the $\psi_i(x_j)$ in $\Q_i$. Since this holds for each $i$, this shows that $\Psi$ is quasimedian.

Combining this with Theorem~\ref{thm:bbfqie}, we have that the orbit map $\Psi$ is a quasimedian quasiisometric embedding of $H$ in a finite product of quasitrees. Take any quasimedian quasiisometry from $\prod_{i=1}^m\Q_i$ to a CAT(0) cube complex $Y$ (such as the obvious quasiisometry to a product of trees, which is quasimedian by Lemma~\ref{lem:qieimpliesqm}), and compose this with $\Psi$ to produce a quasimedian quasiisometric embedding $H\to Y$. Applying Proposition~\ref{proposition:hagenpetyt2.11} completes the proof. 
\end{proof}

\subsection{Proof of {Proposition~\ref{prop:hptoqg}}}  \label{subsection:hptoqg} 

All that remains for this proof of Theorem~\ref{thm:mcgisacubecomplex} is to establish Proposition~\ref{prop:hptoqg}, which says that the $\psi_i$ map $D$--hierarchy paths in $\mcg(S)$ to uniform unparametrised quasigeodesics in $\Q_i$. Recall that $\ppi_h=\mfp_{h\gamma_i}\pi_{hU_i}:H\to h\gamma_i$, where $\gamma_i$ is a quasigeodesic in $\C U_i$.

Let us start with an important observation that is a variant of the Behrstock inequality.

\begin{lemma} \label{lem:behrstockaxis}
There is a constant $\nu=\nu(E,\lambda)\geq10E$ such that if $h,h'\in H$ satisfy $hU_i=h'U_i$, but $h\gamma_i\neq h'\gamma_i$, then for any $z\in H$, at most one of the quantities 
\[
\dist_{h\gamma_i}(\ppi_h(z),\mfp_{h\gamma_i}(h'\gamma_i)) \hspace{5mm} \text{and} \hspace{5mm} \dist_{h'\gamma_i}(\ppi_{h'}(z),\mfp_{h'\gamma_i}(h\gamma_i))
\]
is greater than $\nu$; moreover, $\dist_{h'\gamma_i}(\ppi_{h'}(h),\mfp_{h'\gamma_i}(h\gamma_i))\leq\nu$.
\end{lemma}

\begin{proof} 
The first statement follows from the proof of \cite[Prop.~3.1]{bestvinabrombergfujiwara:proper}, because $h\gamma_i$ and $h'\gamma_i$ are $\lambda$--quasigeodesics and the bounded set $\pi_{h U_i}(z)$ is quasiconvex.
For the ``moreover'' statement, note that the fact that $\gamma_i$ was chosen to minimise $\dist_{\C U_i}(\pi_{U_i}(1),\gamma_i)$ implies that the geodesic in $\C hU_i$ from $\ppi_h(h)$ to $\ppi_{h'}(h)$ lies in a neighbourhood of the union of the geodesics from $\pi_{hU_i}(h)$ to $\ppi_h(h)$ and $\ppi_{h'}(h)$. In particular, $\mfp_{h'\gamma_i}(\ppi_h(h))\in\mfp_{h'\gamma_i}(h\gamma_i)$ lies close to $\ppi_{h'}(h)$.
\end{proof}

The map $\Psi:H\to\prod_{i=1}^m\Q_i$ is equivariant, so we can assume that the $D$--hierarchy path $\beta:\{0,\dots,T\}\to H$ has $\beta(0)=1$. Write $g=\beta(T)$.

\begin{lemma} \label{lemma:lambdad}
There are constants $\lambda'\ge\lambda+\xi$ and $D'\ge\lambda'(D+50\lambda')$, defined in terms of $\lambda$ and $E$ only, such that, for any $h\in H$, 
\begin{enumerate}
\enumlabel[(\ref{lemma:lambdad}.1)] the restriction $\mfp_{h\gamma_i}|_{\C hU_i}$ is $\lambda'$--coarsely Lipschitz,    \label{lemma:lambdad:lambda}
\enumlabel[(\ref{lemma:lambdad}.2)] $\ppi_h\beta$ is an unparametrised $D'$--quasigeodesic from $\ppi_h(1)$ to $\ppi_h(g)$. \label{lemma:lambdad:d}
\end{enumerate}
\end{lemma}

\begin{proof}
By virtue of being a $\lambda$--quasigeodesic, $h\gamma_i$ is quasiconvex in the $E$--hyperbolic space $\C hU_i$, which provides $\lambda'$. The existence of $D'$ is a consequence of this and the fact that, by definition of $\beta$ being a $D$--hierarchy path, $\pi_{U_i}\beta$ is an unparametrised $D$--quasigeodesic. Since $D$ and $\lambda'$ depend only on $\lambda$ and $E$, so does $D'$.
\end{proof}

For a constant $t\geq100D'$ and mapping classes $x,y\in H$, we shall write 
\[\rel^{\gamma_i}_t(x,y)=\left\{h\gamma_i\in H\cdot\{\gamma_i\} \hsp:\hsp \dist_{h\gamma_i}(\ppi_h(x),\ppi_h(y))\geq t\right\}.\] 
As a consequence of \ref{lemma:lambdad:d}, if $h\gamma_i\in\rel^{\gamma_i}_{100D'}(1,g)$ then there exists a minimal $a_h\in\{0,\dots,T\}$ such that $\dist_{h\gamma_i}(\ppi_h\beta(a_h),\ppi_h(1))\geq 2D'$, and similarly there is a maximal $b_h$ with $\dist_{h\gamma_i}(\ppi_h\beta(b_h),\ppi_h(g))\geq2D'$. Note that $a_h$ must be strictly less than $b_h$, and also that the restricted paths $\ppi_h\beta|_{[0,a_h]}$ and $\ppi_h\beta|_{[b_h,T]}$ are $10D'$--coarsely constant. Let us write $x_h=\beta(a_h)\in H$ and $y_h=\beta(b_h)\in H$. Furthermore, write $\beta_h=\beta|_{[a_h,b_h]}\subset H$. See Figure~\ref{figure:xhyh}.

\begin{figure}[ht]
\includegraphics[width=14cm]{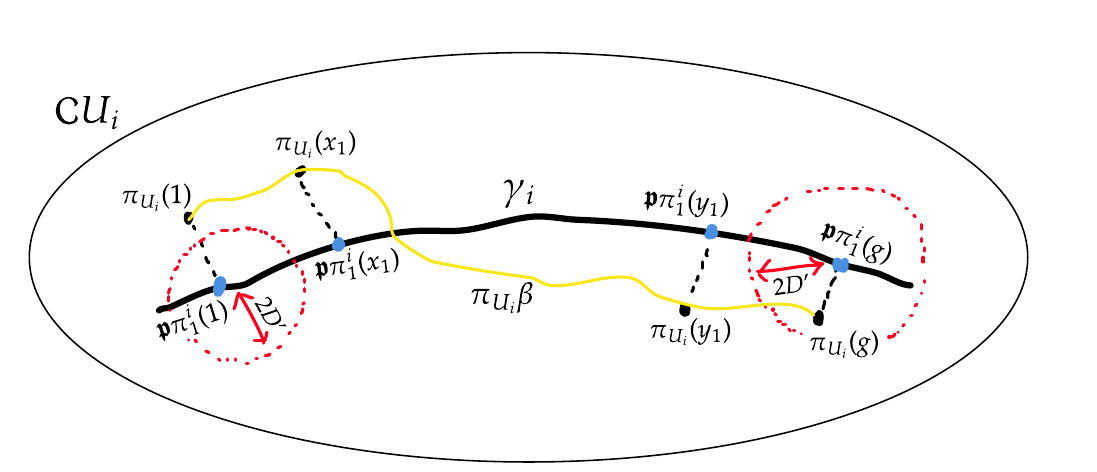}\centering
\caption{Construction of $x_h$ and $y_h$, illustrated with $h=1$. They are the first and last points of $\beta$ that are not mapped into one of the red circles by $\ppi_1$.} \label{figure:xhyh}
\end{figure}

\subsubsection*{\emph{\textbf{The ordering on $\rel^{\gamma_i}_{100D'}(1,g)$.}}} ~

Recall that $D\geq10\nu\geq100E\geq100$, where $\nu$ is the constant coming from Lemma~\ref{lem:behrstockaxis}.

Because $H$ preserves the colouring of $\s$, we have that either $hU_i=h'U_i$ or $hU_i\trans h'U_i$ for any $h,h'\in H$. Because of \ref{lemma:lambdad:lambda}, the lower bound on the value of $D'$ ensures that $hU_i\in\rel_{10\nu}(1,g)$ whenever $h\gamma_i\in\rel^{\gamma_i}_{100D'}(1,g)$. Thus, if $h\gamma_i$ and $h'\gamma_i$ are elements of $\rel^{\gamma_i}_{100D'}(1,g)$ with $hU_i\neq h'U_i$, then we may assume that $hU_i<h'U_i$ in the ordering from Definition~\ref{definition:ordering}, and we then set $h\gamma_i<h'\gamma_i$. Otherwise, $hU_i=h'U_i$, and then Lemma~\ref{lem:behrstockaxis} tells us (perhaps after swapping $h$ and $h'$) that $\ppi_h(g)$ is $\nu$--close to $\mfp_{h\gamma_i}(h'\gamma_i)$, and $\ppi_{h'}(1)$ is $\nu$--close to $\mfp_{h'\gamma_i}(h\gamma_i)$. In this case, we write $h\gamma_i<h'\gamma_i$. 

By construction, if $h\gamma_i<h'\gamma_i$, then we get an ordering of integers $b_h<a_{h'}$. As noted earlier, we also have that $a_h<b_h$. Therefore, the set $\{a_h,b_h : h\gamma_i\in\rel^{\gamma_i}_{100D'}(1,g)\}$ is naturally an ordered subset of $\{0,\dots,T\}$.

\begin{lemma} \label{lem:totalorder}
$<$ is a total order on $\rel^{\gamma_i}_{100D'}(1,g)$.
\end{lemma}

\begin{proof}
It suffices to check that $h\gamma_i<h'\gamma_i$, $h'\gamma_i<h''\gamma_i$ implies $h\gamma_i<h''\gamma_i$ in the case where $hU_i=h'U_i=h''U_i$, as the other cases follow from the fact that $\rel_{10\nu}(1,g)$ is totally ordered. We do this by contradiction, referring to Figure~\ref{fig:order}.

If the implication does not hold, then $\ppi_h(g)$ is $\nu$--far from $\mfp_{h\gamma_i}(h''\gamma_i)$, so $\ppi_{h''}(g)$ is $\nu$--close to $\mfp_{h''\gamma_i}(h\gamma_i)$ by Lemma~\ref{lem:behrstockaxis}. It follows that $\mfp_{h''\gamma_i}(h\gamma_i)$ is $20\nu$--far from both $\ppi_{h''}(1)$ and $\mfp_{h''\gamma_i}(h'\gamma_i)$, the former by $100D'$--relevance of $h''\gamma_i$ and then the latter by the fact that $h'\gamma_i<h''\gamma_i$. Applying Lemma~\ref{lem:behrstockaxis} with $z=1$ now shows that $\ppi_h(1)$ is $\nu$--close to $\mfp_{h\gamma_i}(h''\gamma_i)$. Just as when we considered $h''\gamma_i$, this tells us that $\mfp_{h\gamma_i}(h''\gamma_i)$ is $20\nu$--far from $\mfp_{h\gamma_i}(h'\gamma_i)$, this time by $100D'$--relevance of $h\gamma_i$ and the fact that $h\gamma_i<h'\gamma_i$. 

Now let $x\in H$ be any point such that $\ppi_{h'}(x)$ is $\nu$--far from both $\mfp_{h'\gamma_i}(h\gamma_i)$ and $\mfp_{h'\gamma_i}(h''\gamma_i)$. Then $\ppi_h(x)$ and $\ppi_{h''}(x)$ are, respectively, $\nu$--close to $\mfp_{h\gamma_i}(h'\gamma_i)$ and $\mfp_{h''\gamma_i}(h'\gamma_i)$, by Lemma~\ref{lem:behrstockaxis}. Comparing with the preceding paragraph, we find that $\ppi_h(x)$ is $10\nu$--far from $\mfp_{h\gamma_i}(h''\gamma_i)$, and simultaneously $\ppi_{h''}(x)$ is $10\nu$--far from $\mfp_{h''\gamma_i}(h\gamma_i)$, which contradicts Lemma~\ref{lem:behrstockaxis}.
\end{proof}

\begin{figure}[ht] 
\includegraphics[width=12cm]{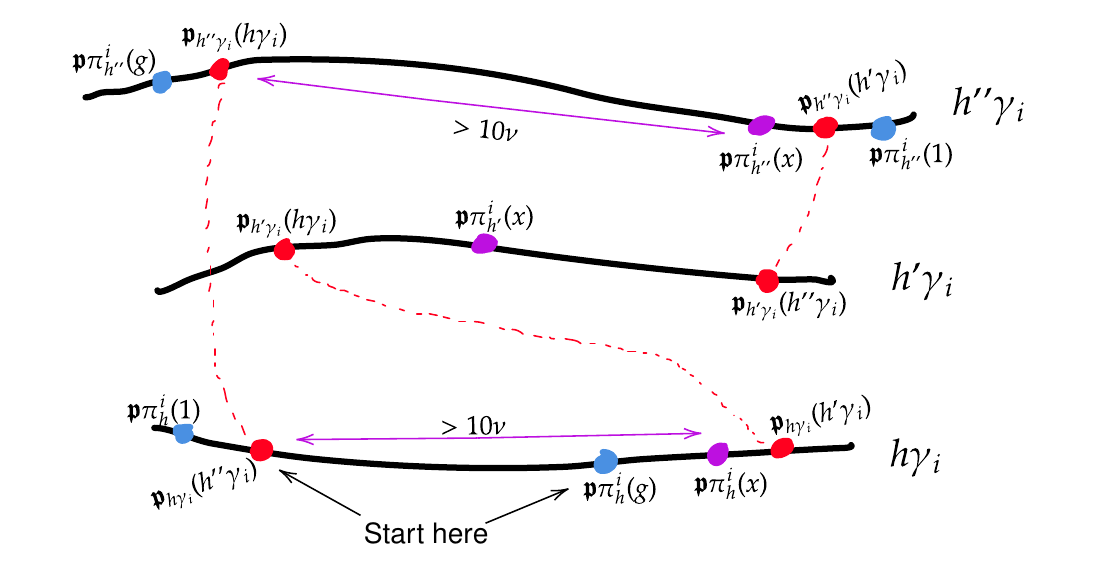}\centering
\caption{The contradiction obtained in the proof of Lemma~\ref{lem:totalorder}.} \label{fig:order}
\end{figure}

Recall that $\beta_h$ denotes the subpath $\beta|_{[a_h,b_h]}$, and $\ppi_h\beta_h$ is an unparametrised $D'$--quasigeodesic. Moreover, $\ppi_h(x)$ is $D'$--far from both $\ppi_h(1)$ and $\ppi_h(g)$ for all $x\in\beta_h$. By definition, $\psi_i(x)=\ppi_x(x)$ for all $x\in H$.

\begin{lemma} \label{lem:closeonbigguys}
If $h\gamma_i\in\rel^{\gamma_i}_{100D'}(1,g)$, then we have $\dist_{\Q_i}(\psi_i(x),\ppi_h(x))\leq6K$ for any $x\in\beta_h$. In particular, $\psi_i\beta_h\subset\N_{6K}(h\gamma_i)\subset\Q_i$ is an unparametrised $(D'+12K)$--quasigeodesic.
\end{lemma}

\begin{proof}
By the distance formula for quasitrees of metric spaces, Theorem~\ref{thm:bbfdf}, if the lemma does not hold then there is some $h'\gamma_i$ such that 
\begin{align}
\dist_{h'\gamma_i}\big(\mfp_{h'\gamma_i}\psi_i(x),\mfp_{h'\gamma_i}(\ppi_h(x))\big)\geq K. \label{eq:proofofcloseonbigguys}
\end{align}
Clearly $x\gamma_i$ and $h\gamma_i$ cannot be equal, for then $\psi_i(x)=\ppi_h(x)$.

If $h\gamma_i=h'\gamma_i$, then we have $\dist_{h\gamma_i}(\mfp_{h\gamma_i}\psi_i(x),\ppi_h(x))\geq K$. In this case, $xU_i$ and $hU_i$ cannot differ, for then $\mfp_{h\gamma_i}\psi_i(x)=\mfp_{h\gamma_i}(\rho^{xU_i}_{hU_i})$ is $2\lambda'E$--close to $\ppi_h(x)$ by inequality~\eqref{eq:nearrho} and \ref{lemma:lambdad:lambda}. However, $xU_i=hU_i$ gives a contradiction with the second statement of Lemma~\ref{lem:behrstockaxis}. Thus $h'\gamma_i\neq h\gamma_i$ and $h\gamma_i\neq x\gamma_i$.

\begin{claim} \label{claim:xh'close}
$\dist_{h\gamma_i}(\ppi_h(x),\mfp_{h\gamma_i}(h'\gamma_i))\leq\xi+\nu+2\lambda'E.$
\end{claim}

\begin{claimproof}
If $h'\gamma_i=x\gamma_i$ and $h'U_i=hU_i$, then the claim is given by Lemma~\ref{lem:behrstockaxis}. If $h'\gamma_i=x\gamma_i$ but $h'U_i\neq hU_i$, then $\mfp_{h\gamma_i}(h'\gamma_i)=\mfp_{h\gamma_i}(\rho^{xU_i}_{hU_i})$ is $2\lambda'E$--close to $\ppi_h(x)$ by inequality~\eqref{eq:nearrho} and~\ref{lemma:lambdad:lambda}. 

On the other hand, if $h'\gamma_i\neq x\gamma_i$, then since the set $H\cdot\{\gamma_i\}$ satisfies \ref{projection0} with constant $\xi$, we have $\dist_{h'\gamma_i}(\mfp_{h'\gamma_i}(h\gamma_i),\mfp_{h'\gamma_i}(x\gamma_i))\geq K-2\xi$. It follows from \ref{projection1} that $\dist_{h\gamma_i}(\mfp_{h\gamma_i}(h'\gamma_i),\mfp_{h\gamma_i}(x\gamma_i))\leq\xi$. If $xU_i=hU_i$ then the second statement of Lemma~\ref{lem:behrstockaxis} shows that $\dist_{h\gamma_i}(\ppi_h(x),\mfp_{h\gamma_i}(x\gamma_i)\leq\nu$, and we conclude by applying the triangle inequality. The alternative is that $xU_i\neq hU_i$, but then inequality~\eqref{eq:nearrho} shows that $\ppi_h(x)$ is $2\lambda' E$--close to $\mfp_{h\gamma_i}(\rho^{xU_i}_{hU_i})=\mfp_{h\gamma_i}(x\gamma_i)$, and we again apply the triangle inequality. 
\end{claimproof}

As noted before the proof, $\ppi_h(x)$ is $D'$--far from both $\ppi_h(1)$ and $\ppi_h(g)$, so by the choice of $D'$, Claim~\ref{claim:xh'close} ensures that these are both $\frac{D'}{2}$--far from $\mfp_{h\gamma_i}(h'\gamma_i)$. 

\begin{claim} \label{claim:1ghclose}
Both $\ppi_{h'}(1)$ and $\ppi_{h'}(g)$ are $\frac{D'}{2}$--close to $\mfp_{h'\gamma_i}(h\gamma_i)$.
\end{claim} 

\begin{claimproof}
If $hU_i=h'U_i$ then this is just Lemma~\ref{lem:behrstockaxis}. If $hU_i\neq h'U_i$, then both $\pi_{hU_i}(1)$ and $\pi_{hU_i}(g)$ must be $E$--far from $\rho^{h'U_i}_{hU_i}$, by \ref{lemma:lambdad:lambda}. The Behrstock inequality for the subsurfaces $hU_i$ and $h'U_i$ then gives that both $\pi_{h'U_i}(1)$ and $\pi_{h'U_i}(g)$ are $E$--close to $\rho^{hU_i}_{h'U_i}$. But $\mfp_{h'\gamma_i}(h\gamma_i)=\mfp_{h'\gamma_i}(\rho^{hU_i}_{h'U_i})$ by definition, so this is enough, again by \ref{lemma:lambdad:lambda}.
\end{claimproof}

To conclude the proof, we show that $\ppi_{h'}(x)$ is $\frac{K}{2}$--far from $\mfp_{h'\gamma_i}(h\gamma_i)$. Together with Claim~\ref{claim:1ghclose}, this will contradict \ref{lemma:lambdad:d}, showing that the proposed $h'\gamma_i$ cannot exist. 

We have $h'\gamma_i\neq h\gamma_i$, and $\mfp_{h'\gamma_i}(\ppi_h(x))$ is an element of $\mfp_{h'\gamma_i}(h\gamma_i)$. Recall that $\mfp_{h'\gamma_i}(h\gamma_i)$ has diameter at most $\xi$. As a consequence of inequality~\eqref{eq:proofofcloseonbigguys}, it thus suffices to put an upper bound on $\dist_{h'\gamma_i}(\mfp_{h'\gamma_i}\psi_i(x),\ppi_{h'}(x))$, for the triangle inequality gives
\begin{align*}
\dist_{h'\gamma_i}(\mfp_{h'\gamma_i}(h\gamma_i),&\ppi_{h'}(x)) 
 \geq \dist_{h'\gamma_i}(\mfp_{h'\gamma_i}(\ppi_h(x)),\ppi_{h'}(x))-\xi \\
&\geq \dist_{h'\gamma_i}(\mfp_{h'\gamma_i}(\ppi_h(x)),\mfp_{h'\gamma_i}\psi_i(x))
    -\dist_{h'\gamma_i}(\mfp_{h'\gamma_i}\psi_i(x),\ppi_{h'}(x)) -\xi \\
&\geq K-\xi-\dist_{h'\gamma_i}(\mfp_{h'\gamma_i}\psi_i(x),\ppi_{h'}(x)).
\end{align*}
If $h'U_i=xU_i$, then this is just the second statement of Lemma~\ref{lem:behrstockaxis}. Otherwise, $\mfp_{h'\gamma_i}\psi_i(x)=\mfp_{h'\gamma_i}(\rho^{xU_i}_{h'U_i})$ is $D'$--close to $\ppi_{h'}(x)$ by inequality~\eqref{eq:nearrho} and \ref{lemma:lambdad:lambda}.
\end{proof}

\subsubsection*{\emph{\textbf{Approximating $\mathrm\psi_i\mathrm\beta$.}}} ~

For a large constant $M>100D'$, to be specified later, enumerate $\rel^{\gamma_i}_M(1,g)=\{h_2\gamma_i,h_4\gamma_i,\dots,h_n\gamma_i\}$, with even subscripts, according to the total order of Lemma~\ref{lem:totalorder}. For even $j$, we have corresponding integers $a_{h_j},b_{h_j}\in\{0,\dots,T\}$ and a subpath $\beta_{h_j}=\beta|_{[a_{h_j},b_{h_j}]}$. Recall that $x_{h_j}$ and $y_{h_j}$ respectively denote $\beta(a_{h_j})$ and $\beta(b_{h_j})$; see Figure~\ref{figure:xhyh}. For the remainder of the proof of Proposition~\ref{prop:hptoqg}, we shall abbreviate $a_{h_j}$ to $a_j$ and $\beta_{h_j}$ to $\beta_j$ etc., omitting the $h$. Set $b_0=0$ and $a_{n+2}=T$. 
For odd $j$, let $\alpha_j=\beta|_{[b_{j-1},a_{j+1}]}$, so that we have a decomposition of $\beta$ as the concatenation $\alpha_1\beta_2\alpha_3\dots\beta_n\alpha_{n+1}$. 

\begin{lemma} \label{lem:fellowtravel}
Each $\psi_i\alpha_j$ is a quasigeodesic with constant independent of $M$.
\end{lemma}

\begin{proof}
Let $\hat\alpha_j$ be a geodesic in $\Q_i$ from $\psi_i(y_{j-1})$ to $\psi_i(x_{j+1})$. It suffices to show that $\psi_i\alpha_j$ fellow-travels $\hat\alpha_j$ with constant independent of $M$.

By \cite[Prop.~4.11]{bestvinabrombergfujiwara:constructing}, there is a constant $q_1=q_1(K,E,\lambda)$ such that if $h\gamma_i\in\rel^{\gamma_i}_{100D'}(1,g)$ satisfies $h_{j-1}\gamma_i<h\gamma_i<h_{j+1}\gamma_i$ for some odd $j$, then
\[
\hat\alpha_j \text{ comes } q_1\text{--close to } \mfp_{h\gamma_i}\psi_i(y_{j-1}).
\]
We know from Lemma~\ref{lem:closeonbigguys} that $\dist_{\Q_i}(\ppi_{h_{j-1}}(y_{j-1}),\psi_i(y_{j-1}))\le6K$, and \cite[Cor.~4.10]{bestvinabrombergfujiwara:constructing} states that $\mfp_{h\gamma_i}:\Q_i\to h\gamma_i$ is coarsely Lipschitz. Because $\ppi_{h_{j-1}}(y_{j-1})\in h_{j-1}\gamma_i$, this means that there is a constant $q_2=q_2(K,E,\lambda)$ such that 
\[
\dist_{\Q_i}(\mfp_{h\gamma_i}(h_{j-1}\gamma_i),\mfp_{h\gamma_i}\psi_i(y_{j-1})) \hsp\le\hsp q_2.
\]

The construction of $x_h$ dictates that $\ppi_h(x_h)$ be $10D'$--close to $\ppi_h(1)$, which itself is $\nu$--close to $\mfp_{h\gamma_i}(h_{j-1}\gamma_i)$ because of how we defined the ordering on $\rel^{\gamma_i}_{100D'}(1,g)$. That is, 
\[
\dist_{\Q_i}(\mfp_{h\gamma_i}(h_{j-1}\gamma_i),\ppi_h(x_h)) \hsp\le\hsp 10D' +\nu.
\]
Also, Lemma~\ref{lem:closeonbigguys} shows that 
\[
\dist_{\Q_i}(\psi_i(x_h),\ppi_h(x_h)) \hsp\le\hsp 6K. 
\]

Combining these inequalities, we find that $\hat\alpha_j$ comes $(q_1+q_2+10D'+\nu+6K)$--close to $\psi_i(x_h)$. A similar argument proves that it comes just as close to $\psi_i(y_h)$. Recall that $h\gamma_i$ is $100D'$--relevant, and in particular that $\psi_i\beta_h=\psi_i\alpha_j|_{[a_h,b_h]}$ is a quasigeodesic, by Lemma~\ref{lem:closeonbigguys}. The Morse lemma now tells us that $\psi_i\alpha_j|_{[a_h,b_h]}$ fellow-travels a subgeodesic of $\hat\alpha_j$, with constant independent of~$M$.

On the other hand, suppose that $x,y\in H$ have $\rel^{\gamma_i}_{100D'}(x,y)=\varnothing$. For any $h\in H$ we have $\mfp_{h\gamma_i}\psi_i(x)\in\mfp_{h\gamma_i}(x\gamma_i)$, and similarly with $y$ in place of $x$, so in this case we must have $\dist_{h\gamma_i}(\mfp_{h\gamma_i}\psi_i(x),\mfp_{h\gamma_i}\psi_i(y))\leq100D'+2\xi<K$ for all $h\in H$. The distance formula for quasitrees of metric spaces, Theorem~\ref{thm:bbfdf}, therefore shows that $\psi_i(x)$ and $\psi_i(y)$ are $6K+\delta$--close, where $\delta$ is the constant from inequality~\eqref{inequality:changedistance}. Taking this together with the conclusion of the previous paragraph completes the proof.
\end{proof}

We have shown that $\psi_i\beta$ decomposes as a concatenation of quasigeodesics. We use a local-to-global statement to conclude that $\psi_i\beta$ is itself a quasigeodesic. The proof is due to Hsu--Wise \cite{hsuwise:cubulating}, but the exact statement appears as \cite[Lem.~4.3]{hagenwise:cubulating:irreducible}.

\begin{lemma} \label{lem:hagenwise}
For any constants $k,\lambda$ and any function $f$, there is a constant $L_0$ satisfying the following. Suppose that $P$ is a path in a $k$--hyperbolic space that decomposes as a concatenation $P=\alpha_1\beta_2\alpha_3\dots\beta_n\alpha_{n+1}$ of $\lambda$--quasigeodesics. Suppose further that the~sets 
\[
\N_{3k+r}(\beta_j)\cap\beta_{j\pm2} \hsp\text{ and }\hsp \N_{3k+r}(\beta_j)\cap\alpha_{j\pm1}
\]
all have diameter at most $f(r)$. If every $\beta_j$ has length at least $L_0$, then $P$ is a uniform quasigeodesic.
\end{lemma}

We need one more lemma in order to be able to apply this result. 

\begin{lemma} \label{lem:coarseintersections}
There is a function $f$, independent of $M$, such that the subsets 
\[
\N_r(\psi_i\beta_j)\cap\N_r(\psi_i\beta_{j\pm2}) \hspace{5mm} \text{and} \hspace{5mm} \psi_i\alpha_j\cap\N_r(\psi_i\beta_{j\pm1})
\]
of $\Q_i$ have diameters bounded by $f(r)$ for all $r$.
\end{lemma}

\begin{proof}
\cite[Cor.~4.10]{bestvinabrombergfujiwara:constructing} shows that if $h'\gamma_i\neq h\gamma_i$, then the closest-point projection of $h'\gamma_i$ to $h\gamma_i$ inside $\Q_i$ is coarsely equal to $\mfp_{h\gamma_i}(h'\gamma_i)$, which has diameter at most $\xi$. Thus $\N_r(h\gamma_i)\cap\N_r(h'\gamma_i)$ has diameter bounded in terms of $r$. Now recall from Lemma~\ref{lem:closeonbigguys} that $\psi_i\beta_j$ is contained in the $6K$--neighbourhood of $h_j\gamma_i$. This is enough for the first intersection, because \cite[Cor.~4.10]{bestvinabrombergfujiwara:constructing} says that closest-point projection to $h\gamma_i$ in $\Q_i$ is coarsely Lipschitz.

This also bounds the Hausdorff-distance between $\mfp_{h_{j+1}\gamma_i}(h_{j-1}\gamma_i)$ and the closest-point projection of $\psi_i(x_{j+1})$ to $h_{j+1}\gamma_i$ as follows. The latter of the two is a bounded distance from $\mfp_{h_{j+1}\gamma_i}\psi_i(x_{j+1})$, which itself is $6K$--close to $\ppi_{h_{j+1}}(x_{j+1})$ by Lemma~\ref{lem:closeonbigguys}, and this point is $10D'$--close to $\mfp_{h_{j+1}\gamma_i}(h_{j-1}\gamma_i)$ by construction of the ordering. 

Moreover, $\ppi_{h_{j-1}}(y_{j-1})$ is $6K$--close to $\mfp_{h_{j-1}\gamma_i}\psi_i(y_{j-1})$, again because of Lemma~\ref{lem:closeonbigguys}, and this bounds the Hausdorff-distance between $\mfp_{h_{j+1}\gamma_i}(h_{j-1}\gamma_i)$ and the closest-point projection of $\psi_i(y_{j-1})$ to $h_{j+1}\gamma_i$, by \cite[Cor.~4.10]{bestvinabrombergfujiwara:constructing} as before. 

According to Lemma~\ref{lem:fellowtravel}, $\psi_i\alpha_j$ is a quasigeodesic from $\psi_i(y_{j-1})$ to $\psi_i(x_{j+1})$, and we have just seen that the closest-point projections to $h_{j+1}\gamma_i$ of both of these endpoints are at a bounded distance from $\mfp_{h_{j+1}\gamma_i}(h_{j-1}\gamma_i)$. This shows that the closest-point projection to $h_{j+1}\gamma_i$ of the whole quasigeodesic $\psi_i\alpha_j$ is a coarse point. A similar argument works for $h_{j-1}\gamma_i$. The bound for the second intersection follows because Lemma~\ref{lem:closeonbigguys} shows that $\psi_i\beta_{j\pm1}\subset\N_{6K}(h_{j\pm1}\gamma_i)$.
\end{proof}

\subsubsection*{\emph{\textbf{Conclusion of the proof of Proposition~\ref{prop:hptoqg}.}}} ~

In light of Lemmas~\ref{lem:closeonbigguys},~\ref{lem:fellowtravel}, and~\ref{lem:coarseintersections}, the conditions of Lemma~\ref{lem:hagenwise} are met (up to parametrisation) by the path $\psi_i\beta=(\psi_i\alpha_1)(\psi_i\beta_2)\dots(\psi_i\alpha_{n+1})$, with all constants independent of $M$. There is therefore a constant $L_0$, independent of $M$, such that if every $\psi_i\beta_j$ has length at least $L_0$ then $\psi_i\beta$ is an unparametrised quasigeodesic.

Since $L_0$ and $K$ are independent of $M$, we can now take $M=L_0+20K$, so that $\dist_{h_j\gamma_i}(\ppi_{h_j}(1),\ppi_{h_j}(g))\ge L_0+20K$ for each $j$. Since $h_j\gamma_i$ is totally geodesically embedded in $\Q_i$ \cite[Lem.~4.2]{bestvinabrombergfujiwara:constructing}, the same inequality holds in $\Q_i$. 

The length of $\psi_i\beta_j$ is at least the distance between its endpoints $\psi_i(x_j)$ and $\psi_i(y_j)$, which are $6K$--close to $\ppi_{h_j}(x_j)$ and $\ppi_{h_j}(y_j)$, respectively, by Lemma~\ref{lem:closeonbigguys}. In turn, these are $10D'$--close to $\ppi_{h_j}(1)$ and $\ppi_{h_j}(g)$, respectively. Thus the length of $\psi_i\beta_j$ is at least $\dist_{\Q_i}(\ppi_{h_j}(1),\ppi_{h_j}(g))-20D'-12K\geq L_0$. Hence $\psi_i\beta$ is an unparametrised quasigeodesic, with constant $q$ bounded in terms of the various constants of this section, all of which are defined in terms of $E$ and $\lambda$ only. This completes the proof of Proposition~\ref{prop:hptoqg}. \hfill\qed

\begin{remark}
Although the quasilines produced in \cite{bestvinabrombergfujiwara:proper} in many ways behave similarly to the domains of a hierarchically hyperbolic structure, they do not actually provide one for $\mcg(S)$, because that would imply that $\mcg(S)$ were virtually abelian \cite[Cor.~4.2]{petytspriano:unbounded}. To get such a structure, one would have to include the cone-offs of curve graphs by the axes.
\end{remark}

\section{Proof using hyperbolic cones} \label{section:buyalo}

In this section, we use a construction of Buyalo--Dranishnikov--Schroeder \cite{buyalodranishnikovschroeder:embedding} to prove Theorem~\ref{mthm:hyp_asdim_ccc}, which states that hyperbolic spaces with finite asymptotic dimension are quasiisometric to CAT(0) cube complexes. We then apply this to show that colourable hierarchically hyperbolic groups admit quasimedian quasiisometric embeddings in finite products of CAT(0) cube complexes. In the case of mapping class groups, Hume \cite{hume:embedding} used a related construction of Buyalo \cite{buyalo:capacity} to quasiisometrically embed $\mcg(S)$ in a finite product of trees. However, taking care of medians allows us to apply Proposition~\ref{proposition:hagenpetyt2.11} to obtain a quasiisometry to a CAT(0) cube complex, establishing Theorem~\ref{thm:colourable_HHGs_quasicubical}, which implies Theorem~\ref{thm:mcgisacubecomplex}.

\subsection{Embedding hyperbolic spaces in finite products of trees} \label{subsection:bds_embedding}

Here we summarise the construction of an embedding of a hyperbolic space into a finite product of (not necessarily locally finite) trees as described in \cite[\S7--9]{buyalodranishnikovschroeder:embedding}. Following \cite{bonkschramm:embeddings}, we say that a hyperbolic space $X$ is \emph{visual} if for some basepoint $x_0\in X$ there is a constant $D$ such that every $x\in X$ lies on a $D$--quasigeodesic ray emanating from~$x_0$.

\begin{lemma} \label{lem:visualisation}
Every hyperbolic space $X'$ admits a median isometric embedding in a visual hyperbolic space $X$ with \emph{asymptotic dimension} $\asdim X=\max\{1,\asdim X'\}$. 
\end{lemma}

\begin{proof}
Given $X'$, let $X$ be the hyperbolic space obtained by attaching a ray $r_x=[0,\infty)$ to each point $x\in X'$. Fix $x_0\in X'\subset X$. We see that $X$ is visual by concatenating a geodesic from $x_0$ to $x$ with the geodesic ray $r_x$. The upper bound on asymptotic dimension is given by \cite[Thm~25]{belldranishnikov:asymptotic}, and the lower bound is given by \cite[Prop.~23]{belldranishnikov:asymptotic}. The inclusion map $X'\hookrightarrow X$ is median and isometric.
\end{proof} 

\subsubsection*{\textbf{\emph{The hyperbolic cone.}}} ~

We first construct the \emph{hyperbolic cone} on a complete, bounded metric space. The exact formulation we follow is one variant from a broader circle of ideas \cite{dranishnikov:onhypersphericity,buyalo:capacity,langschlichenmaier:nagata}. 

Let $Z$ be a complete, bounded metric space. For a point $z\in Z$, write $B_Z(z,r)$ for the open $r$--ball centred on $z$. For a subset $Y\subset Z$, we similarly write $B_Z(Y,r)=\bigcup_{y\in Y}B_Z(y,r)$. Recall that $Y\subset Z$ is said to be \emph{$r$--separated} if $\dist(y,y')\ge r$ for all $y,y'\in Y$, and $Y$ is an \emph{$r$--net} if $B_Z(Y,r)=Z$. Any maximal $r$--separated set is automatically an $r$--net. 

After rescaling $Z$, we assume that $\diam(Z)<1$. We define the hyperbolic cone on $Z$, denoted $\cone Z$, as follows. Fix any positive constant $r\le\frac16$. For each $k\ge0$, let $V_k$ be a maximal $r^k$--separated subset of $Z$. We associate with each $v\in V_k$ the ball $B(v)=B_Z(v,2r^k)$. Note that $V_0$ is a singleton. We write $o$ for the element of $V_0$; it has $B(o)=Z$.

Let $V=\bigsqcup_{k\ge0}V_k$. The vertex set of $\cone Z$ is $V$, and two vertices $v_1\in V_{k_1}$ and $v_2\in V_{k_2}$ are joined by an edge if either of the following holds.
\begin{itemize}
\item   $k_1=k_2$ and the closed balls $\bar B(v_1)$ and $\bar B(v_2)$ intersect. 
\item   $k_1=k_2-1$ and $B(v_2)\subset B(v_1)$. 
\end{itemize}
The \emph{level} $\ell(v)$ of a vertex $v\in V$ is the unique $k$ such that $v\in V_k$. We call $r$ the \emph{parameter} of the cone.

\begin{proposition}[{\cite[Prop.~2.1]{bourdonpajot:cohomologie},\cite[Thm~7.1]{buyalodranishnikovschroeder:embedding}}] \label{prop:cone_is_hyperbolic}
If $Z$ is a complete and bounded metric space, then $\cone Z$ is hyperbolic, and the hyperbolicity constant depends only on $\diam Z$ and the parameter $r$. If $X$ is a visual hyperbolic space, then there exists a quasiisometric embedding $X\to\cone(\partial X)$. 
\end{proposition}

Using \cite[Prop.~5.6]{bonkschramm:embeddings}, it is straightforward to see that if $X$ is visual then the embedding $X\to\cone(\partial X)$ is also coarsely onto, making it a quasiisometry. We shall not need this fact, though. Also see \cite[Thm~8.2]{bonkschramm:embeddings}.

\subsubsection*{\textbf{\emph{Embedding cones in products of trees.}}}  ~

Following \cite[\S8,9]{buyalodranishnikovschroeder:embedding}, we now describe how to construct the product of trees. We start by building a sequence of coloured open covers of our bounded metric space $Z$. Each colour gives rise to a tree.

The notion of \emph{capacity dimension} was introduced by Buyalo in \cite{buyalo:asymptotic}, and the reader is referred there for a definition. We say that a collection of subsets of $Z$ is \emph{disjoint} if no two of its members intersect. An $n$--\emph{colouring} of a collection $\cal V$ of subsets of $Z$ is a finite decomposition $\cal V=\bigcup_{c\in C}\cal V^c$, with $|C|=n$, such that each $\cal V^c$ is disjoint. Note that the $\cal V^c$ need not form a partition. 

\begin{proposition}[{\cite[Prop.~2.3]{buyalo:capacity}, \cite[Thm~8.2]{buyalodranishnikovschroeder:embedding}}] \label{prop:coloured_covering}
Let $Z$ be a complete, bounded metric space with capacity dimension $n$. There is a constant $\eps\in(0,1)$ such that for any sufficiently small $r\in(0,\frac\eps4)$ there exists a sequence $(\cal U_k=\bigcup_{c\in C}\cal U^c_k)_{k\ge0}$ of $(n+1)$--coloured open covers of $Z$ such that, for any hyperbolic cone on $Z$ with parameter $r$, the following are satisfied.
\begin{enumerate}
\enumlabel[(C1)]    $\sup\{\diam_Z(U):U\in\cal U_k\}<r^k$ for every $k$. Moreover, $\cal U^c_0=\{Z\}$ for all $c\in C$. \label{condition:level_0}
\enumlabel[(C2)]   For every $v\in V_{k+1}$ there exists $U\in\cal U_k$ such that $B(v)\subset U$. \label{condition:parent}
\enumlabel[(C3)]   For every $c\in C$ and for any distinct $U\in \cal U^c_k$ and $U'\in\cal U^c_{k'}$ with $k\le k'$, we have that $B_Z(U',\eps r^{k'})$ is either disjoint from $U$ or is a subset of it. \label{condition:disjointness}
\end{enumerate}
\end{proposition}

Our arguments will not make explicit use of \ref{condition:parent}, though it is used to prove Proposition~\ref{prop:bds_qie}. Define $\cal U^c=\bigsqcup_{k\ge0}(\cal U^c_k\times\{k\})$. Formally, an element of $\cal U^c$ is a pair $(U,k)$, where $U\in\cal U^c_k$, but we shall often abuse notation slightly by just writing $U\in\cal U^c$. We call $k$ the \emph{level} of $U$, and denote it by $\ell(U)$, just as with elements of $V$. Let $\cal U=\bigsqcup_{c\in C}\cal U^c$. See Figure~\ref{fig:covers}.

\begin{figure}[ht]
\includegraphics[width=14cm]{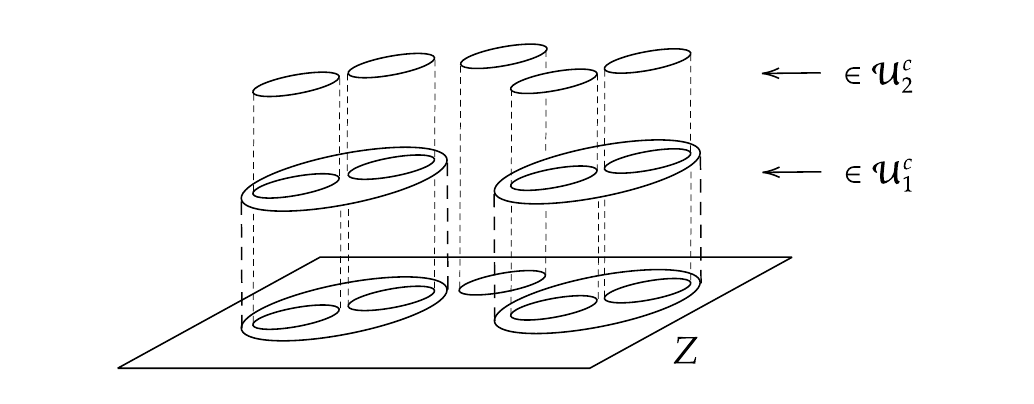}\centering
\caption{Schematic picture of one of the $\cal U^c$.} \label{fig:covers}
\end{figure}

With Proposition~\ref{prop:coloured_covering} in hand, let us now fix a sufficiently small constant $r<\frac17$, a hyperbolic cone on $Z$ with parameter $r$, and a sequence $(\cal U_k)_{k\ge0}$ of coloured covers of $Z$ as above. To improve clarity, let us write $\cal U^c_0=\{o_c\}$. Of course, \ref{condition:level_0} states that, as subspaces of $Z$, we have $o_c=Z$ for all $c$. 

Now, for each colour $c$ we build a rooted tree $T_c$. The vertex set of $T_c$ is $\cal U^c$, and the root is $o_c$. We join vertices $(U,k)$ and $(U',k')$ with $k<k'$ by an edge if $U'\subset U$ as subsets of $Z$ and there is no $(U'',k'')\in\cal U^c$ with $U'\subset U''$ and $k<k''<k'$. 

\begin{figure}[ht]
\includegraphics[width=14cm]{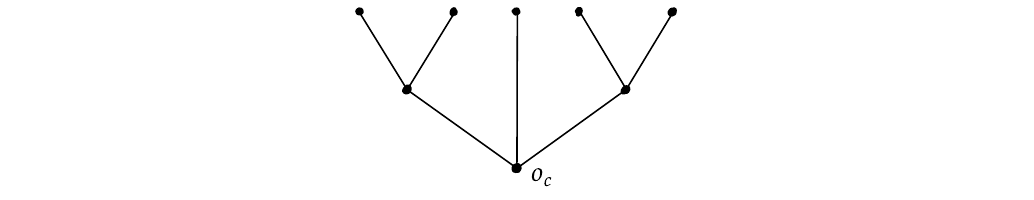}\centering
\caption{The part of $T_c$ corresponding to Figure~\ref{fig:covers}.} \label{fig:tree}
\end{figure}

We are ready to define a map $f_c:\cone Z\to T_c$. We set $f_c(o)=o_c$. For $v\in V_k$ with $k>0$, we define $f_c(v)=U\in\cal U^c$ to be the element of maximal level that has $B(v)\subset U$. This exists because $o_c=Z$, and it is well defined by disjointness of each $\cal U^c_j$. Each point $x\in\cone Z\smallsetminus V$ lies on some edge $vv'$, and we choose $f_c(x)\in\{f_c(v),f_c(v')\}$ arbitrarily.

\begin{proposition}[{\cite[Lem.~9.9, Thm~9.2]{buyalodranishnikovschroeder:embedding}}] \label{prop:bds_qie}
Suppose that $Z$ is a complete, bounded metric space with capacity dimension $n$. The maps $f_c|_V:V\to T_c$ are 2--Lipschitz, and $(f_c)_{c\in C}:\cone Z\to\prod_{c\in C}T_c$ is a quasiisometric embedding of $\cone Z$ in a product of $n+1$ trees.
\end{proposition}

\subsection{The embedding is quasimedian}

We start with a couple of simple preliminary lemmas. With each $U\in\cal U_k$ we associate a subset $\hat U\subset V_k$ by setting $\hat U=\{v\in V_k:B(v)\cap U\neq\varnothing\}$. For two vertices $s$ and $t$ of a tree $T$, we write $[s,t]$ for the unique geodesic between them.

\begin{lemma} \label{lem:hat_map_close}
Suppose that $U\in\cal U^c_k$ and that $v\in \hat U$. We have $f_c(v)\in[o_c,U]\cap B_{T_c}(U,2)$.
\end{lemma}

\begin{proof}
As noted in \cite[\S9.3]{buyalodranishnikovschroeder:embedding}, the fact that $B(v)$ has radius $2r^k$, which is greater than $\sup\{\diam_Z(W):W\in\cal U_j\}$ for all $j\ge k$, implies that $\ell(f_c(v))<k$. Since $B(v)$ intersects $U$, property~\ref{condition:disjointness} of $\cal U$ tells us that $f_c(v)\in[o_c,U]$.

If $\dist_{T_c}(o_c,U)\le2$ then we are done. Otherwise, let $U'$, $U''$ be the vertices of $[o_c,U]$ with $\dist_{T_c}(U,U')=1$ and $\dist_{T_c}(U,U'')=2$. We have a chain of subsets of $Z$ as follows:
\begin{align*}
B(v) &\subset B_Z(U,4r^k) \subset B_Z(U',4r^k) \subset \\
    &\subset B_Z(U',\eps r^{k-1}) \subset B_Z(U',\eps r^{\ell(U')}) \subset U''.
\end{align*}
By definition of $f_c$, we have that $f_c(v)\in\{U',U''\}$.
\end{proof}

\begin{lemma} \label{lem:u_hat_small}
For every $U\in\cal U$, the set $\hat U$ is nonempty and has $\diam_{\cone Z}(\hat U)\le2$.
\end{lemma}

\begin{proof}
Nonemptiness is automatic because $V_k$ is an $r^k$--net. Let $v\in\hat U$. There exists $v^-\in V_{k-1}$ such that $\dist_Z(v^-,v)\le r^{k-1}$. Now, if $z\in B(v')$ for some $v'\in\hat U$, then
\begin{align*}
\dist_Z(z,v^-) &\le \dist_Z(z,U)+\diam_Z(U)+\dist_Z(U,v)+\dist_Z(v,v^-) \\
    &\le 4r^k + r^k + 2r^k + r^{k-1},
\end{align*}
which is less than $2r^{k-1}$ because $r<\frac17$. Thus $B(v')\subset B(v^-)$ for all $v'\in\hat U$, so every $v'\in\hat U$ is joined to $v^-$ by an edge of $\cone Z$. Thus $\diam_{\cone Z}(\hat U)\le2$.
\end{proof}

We can now establish that the map of Proposition~\ref{prop:bds_qie} is quasimedian.

\begin{proposition} \label{prop:bds_quasimedian}
Suppose that $Z$ is a complete, bounded metric space with capacity dimension $n$. The map $(f_c)_{c\in C}:\cone Z\to\prod_{c\in C}T_c$ described in Section~\ref{subsection:bds_embedding} is quasimedian.
\end{proposition}

\begin{proof} 
We must show that each factor map $f_c$ is quasimedian, and it is enough to work with the restriction $f_c|_V$ of $f_c$ to the coarsely dense subset $V\subset\cone Z$. According to Lemma~\ref{lem:partial_quasimorphism}, it suffices to show that $f_c$ is a coarse morphism for the binary operations $\mu_{\cone Z}(\cdot,\cdot,o)$ and $\mu_{T_c}(\cdot,\cdot,o_c)$. Let $x_1$ and $x_2$ be vertices of $\cone Z$, and write $k_i=\ell(x_i)$.

For any $v\in V_k$ and any $j<k$, the fact that $V_j$ is an $r^j$--net implies that there is some $v_j\in V_j$ with $\dist_Z(v,v_j)\le r^j$. For any $z\in B(v)$, we have
\[ 
    \dist_Z(z,v_j) \le \dist_Z(z,v)+\dist_Z(v,v_j) \le 2r^k+r^j < 2r^j.
\]
Hence $B(v)\subset B(v_j)$. Applying this to $x_1$ and $x_2$ shows that we can fix geodesics $\gamma_i=(o=x_{i_0},x_{i_1},\dots,x_{i_{k_i}}=x_i)$, from $o$ to $x_i$ inside $\cone Z$, such that $\ell(x_{i_j})=j$. We have $B(x_{i_j})\subset B(x_{i_{j-1}})$ for all $i,j$.

Recall that for a vertex $x\in V$, the image $f_c(x)$ is defined to be the element $U\in\cal U^c$ of maximal level such that $B(x)\subset U$. Let us write $U_i=f_c(x_i)$ and $U_{i_j}=f_c(x_{i_j})$. If $j_1\le j_2$, then $B(x_{i_{j_2}})\subset B(x_{i_{j_1}})$, so $U_{i_{j_2}}\subset U_{i_{j_1}}$, by property~\ref{condition:disjointness} of $\cal U$. Hence $f_c\gamma_i$ is a monotone map to the unique geodesic $[o_c,U_i]$ in $T_c$. 

The median $\mu_{T_c}(U_1,U_2,o_c)$ is the maximal-level element $U_{12}\in[o_c,U_1]\cap[o_c,U_2]$. In other words, it is the element $U_{12}\in\cal U^c$ of maximal level such that $B(x_1)\cup B(x_2)\subset U_{12}$. Let us write $k_{12}=\ell(U_{12})$.

Let $\delta$ be a hyperbolicity constant for $\cone Z$, which exists by Proposition~\ref{prop:cone_is_hyperbolic}. Fix a geodesic $\gamma_{12}$ in $\cone Z$ from $x_1$ to $x_2$, and define 
\[
    M=\big\{x\in\cone Z:\max\{\dist_{\cone Z}(x,\gamma_1), \dist_{\cone Z}(x,\gamma_2), \dist_{\cone Z}(x,\gamma_{12})\}\le\delta+2\big\} \subset \cone Z.
\]
Because $\cone Z$ is $\delta$--hyperbolic, it is $\delta+2$--hyperbolic. Thus $\mu_{\cone Z}(x_1,x_2,o)\in M$, and the diameter of $M$ is bounded by some constant $D=D(\delta)$. Note that $M\cap\gamma_i\ne\varnothing$.

Let $y_i$ be the unique element of $M\cap\gamma_i$ of maximal level. Because $U_{12}$ contains $B(x_i)$, it intersects $B(x_{i_{k_{12}}})$. Thus $x_{i_{k_{12}}}\in\hat U_{12}$, so Lemma~\ref{lem:u_hat_small} tells us that $\dist_{\cone Z}(x_{1_{k_{12}}},x_{2_{k_{12}}})\le2$. In particular, $x_{i_{k_{12}}}$ is $2$--close to both $\gamma_1$ and $\gamma_2$. It follows from maximality of $y_i$ that $\ell(y_i)\ge\ell(x_{i_{k_{12}}})=k_{12}$. Since $f_c\gamma_i$ is monotone in $[o_c,U_i]$, it follows that $f_c(y_i)\in[f_c(x_{i_{k_{12}}}),U_i]$. On the other hand, Lemma~\ref{lem:hat_map_close} states that $f_c(x_{i_{k_{12}}})\in[o_c,U_{12}]\cap B_{T_c}(U_{12},2)$. Taking these together, we see that 
\[
\dist_{T_c}\big(\mu_{T_c}(f_c(y_1),f_c(y_2),o_c),\ssp U_{12}\big) \le 2.
\]

Because $y_1$, $y_2$, and $m=\mu_{\cone Z}(x_1,x_2,o)$ all lie in $M$, the fact that $f_c|_V$ is 2--Lipschitz (Proposition~\ref{prop:bds_qie}) shows that $\dist_{T_c}(f_c(y_i),f_c(m))\le2D$. To sum up, we have 
\begin{align*}
\dist_{T_c}\big(\mu_{T_c}(f_c(x_1),&f_c(x_2),o_c),\ssp f_c(\mu_{\cone Z}(x_1,x_2,o))\big) \\
    =\hsp&      \dist_{T_c}(U_{12},f_c(m)) \\
    \le\hsp&    \dist_{T_c}\big(\mu_{T_c}(f_c(y_1),f_c(y_2),o_c),\ssp f_c(m)\big) + 2 \\
    \le\hsp&    \dist_{T_c}\big(\mu_{T_c}(f_c(m),f_c(m),o_c), f_c(m)\big) +2+4D \\
    =\hsp&      2+4D,
\end{align*}
because $\mu_{T_c}$ is 1--Lipschitz in each coordinate and $\mu_{T_c}(a,a,b)=a$. 
\end{proof}

\begin{theorem} \label{thm:tfae_hyp_asdim_quasicubical}
If $X$ is a hyperbolic space, then $X$ is quasiisometric to a finite-dimensional CAT(0) cube complex if and only if $X$ has finite asymptotic dimension.
\end{theorem}

\begin{proof}
Suppose that $X$ has finite asymptotic dimension and let $Y$ be the visual hyperbolic space obtained by applying Lemma~\ref{lem:visualisation} to $X$. Proposition~\ref{prop:cone_is_hyperbolic} shows that $Y$ admits a quasiisometric embedding in $\cone(\partial Y)$, and this map is automatically quasimedian by Lemma~\ref{lem:qieimpliesqm}.

As the boundary of a hyperbolic space, $\partial Y$ is complete and bounded \cite[Prop.~6.2]{bonkschramm:embeddings}. Moreover, $\asdim Y\le1+\asdim X$ is finite, so \cite[Prop.~3.6]{mackaysisto:embedding} shows that the capacity dimension of $\partial Y$ is finite. The conditions of Proposition~\ref{prop:bds_qie} are therefore met by $Z=\partial Y$, so $\cone(\partial Y)$ admits a quasimedian quasiisometric embedding in a finite product of trees by Propositions~\ref{prop:bds_qie} and~\ref{prop:bds_quasimedian}. Composing with the embeddings of $X$ in $Y$ and of $Y$ in $\cone(\partial Y)$ yields a quasimedian quasiisometric embedding of $X$ in a finite product of trees, which is a CAT(0) cube complex. We conclude from Proposition~\ref{proposition:hagenpetyt2.11} that $X$ is quasiisometric to a finite-dimensional CAT(0) cube complex.

For the converse, suppose that $X$ is quasiisometric to a finite-dimensional CAT(0) cube complex. Since asymptotic dimension is preserved by quasiisometries \cite[Prop.~22]{belldranishnikov:asymptotic}, the result follows from Wright's theorem that the asymptotic dimension of a CAT(0) cube complex is bounded above by its cubical dimension \cite{wright:finite}.
\end{proof}

\begin{remark}
The proof of Theorem~\ref{thm:tfae_hyp_asdim_quasicubical} implies that, for hyperbolic spaces, the existence of a quasiisometric embedding in a finite product of trees is enough to guarantee the existence of a quasimedian quasiisometric embedding. It would be interesting to have a constructive proof of this fact.
\end{remark}

\subsection{Application to HHGs} \label{subsection:hhs}

We finish this section by applying Theorem~\ref{thm:tfae_hyp_asdim_quasicubical} to colourable hierarchically hyperbolic groups, of which mapping class groups are examples. The resulting Theorem~\ref{thm:colourable_HHG_QC_body} also applies to Artin groups of extra large type \cite[Thm~6.15, Rem.~6.16]{hagenmartinsisto:extra} and extensions of Veech groups \cite[Thm~1.4]{dowdalldurhamleiningersisto:extensions:2}, amongst others.

Hierarchically hyperbolic groups (HHGs) are groups that display nonpositive curvature similar to the hierarchy structure of mapping class groups that was described in Section~\ref{subsection:hierarchy}. The full definition is somewhat technical (see \cite[Def.~1.1]{behrstockhagensisto:hierarchically:2}, \cite[p.4]{petytspriano:unbounded}), so let us just summarise the facts that are relevant for us and not give a complete definition.

An HHG is a pair $(G,\s)$, where $G$ is a group with a fixed finite generating set and $\s$ is a set with a cofinite $G$--action. Moreover, the following hold (\emph{cf.} the list in Section~\ref{subsection:hierarchy}). 
\begin{itemize}
\item   Each $W\in\s$ has an associated uniformly hyperbolic space $\C W$.
\item   There are three ways for elements of $\s$ to be related. One of these relations is called \emph{transversality}, and denoted $W\trans V$. In this situation, there are uniformly bounded subsets $\rho^W_V\subset\C V$ and $\rho^V_W\subset\C W$.
\item   $G$ is a coarse median space.
\end{itemize}

We say that an HHG $(G,\s)$ is \emph{colourable} if there is a finite partition $\s=\bigsqcup_{i=1}^\chi\s_i$ such that $G$ acts by permutations on $\{\s_i:1\le i\le\chi\}$ and any two elements of any one $\s_i$ are transverse. Note that $G$ need not be virtually torsionfree \cite{hughes:lattices}. The important part about the colouring for us is that it gives access to the following.

\begin{proposition}[{\cite[Lem.~3.4, Thm~3.1]{hagenpetyt:projection}}] \label{prop:coloured_HHGs_embed}
If $(G,\s)$ is a colourable HHG, then the projection axioms are satisfied by $(\s_i,\{\rho^W_V:W,V\in\s_i\})$ for each $i$. We can therefore build quasitrees of metric spaces $\C_K\s_i$. For $K$ sufficiently large, there is a quasimedian quasiisometric embedding $G\to\prod_{i=1}^\chi\C_K\s_i$.
\end{proposition}

With this in hand, we can now use Theorem~\ref{thm:tfae_hyp_asdim_quasicubical} to prove Theorem~\ref{thm:colourable_HHGs_quasicubical}, of which Theorem~\ref{thm:mcgisacubecomplex} is a special case.

\begin{theorem} \label{thm:colourable_HHG_QC_body}
Let $(G,\s)$ be a colourable hierarchically hyperbolic group. There is a finite-dimensional CAT(0) cube complex $Q$ with a quasimedian quasiisometry $G\to Q$.
\end{theorem}

\begin{proof}
Proposition~\ref{prop:coloured_HHGs_embed} states that $G$ admits a quasimedian quasiisometric embedding $G\to\prod_{i=1}^\chi\C_K\s_i$ in a finite product of hyperbolic spaces. According to \cite[Cor.~3.3]{behrstockhagensisto:asymptotic}, there is a uniform bound on the asymptotic dimension of the $\C W$ for $W\in\s$, so \cite[Thm~4.24]{bestvinabrombergfujiwara:constructing} shows that each $\C_K\s_i$ has finite asymptotic dimension. Theorem~\ref{thm:tfae_hyp_asdim_quasicubical} now provides finite-dimensional CAT(0) cube complexes $Q_i$ such that $\C_K\s_i$ is quasimedian quasiisometric to $Q_i$. We therefore have a quasimedian quasiisometric embedding of $G$
\[
G\to\prod_{i=1}^\chi\C_k\s_i\to\prod_{i=1}^\chi Q_i
\]
in a finite-dimensional CAT(0) cube complex. The result follows from Proposition~\ref{proposition:hagenpetyt2.11}.
%
%
\end{proof}

\begin{remark}
In the setting of mapping class groups, finiteness of the asymptotic dimensions of the curve graphs $\C W$ was first proved by Bell--Fujiwara \cite{bellfujiwara:asymptotic}, using work of Bowditch \cite{bowditch:tight}. The bound was made explicit by Webb \cite{webb:combinatorics} and has since been improved by Bestvina--Bromberg \cite{bestvinabromberg:onasymptotic}.
\end{remark}

\section{Median-quasiconvexity} \label{section:convexity}

Quasiconvexity is an important concept in the study of hyperbolic spaces, and there have been various generalisations to larger classes of spaces, many of which coincide for mapping class groups. Indeed, the infinite-index \emph{convex cocompact} subgroups of $\mcg(S)$ coincide with the \emph{Morse} subgroups \cite{kim:stable} and the \emph{stable} subgroups \cite{durhamtaylor:convex}. These quasiconvexity conditions are quite restrictive, and such subgroups are always hyperbolic. They can also be characterised as the infinite-index subgroups whose orbits in the curve graph are quasiisometric embeddings \cite{kentleininger:shadows, hamenstadt:word, durhamtaylor:convex, abbottbehrstockdurham:largest}. 

A weaker version of quasiconvexity that is more appropriate for the coarse-median setting is that of median-quasiconvexity, introduced by Bowditch \cite{bowditch:convex}. Recall that a subset $Y$ of a CAT(0) cube complex $Q$ is convex if $\mu(x,y_1,y_2)\in Y$ whenever $y_1,y_2\in Y$.

\begin{definition}[Median-quasiconvexity]
A subset $Y$ of a coarse median space $(X,\mu)$ is $k$--median-quasiconvex if $\mu(x,y_1,y_2)$ is $k$--close to $Y$ whenever $y_1,y_2\in Y$.
\end{definition}

In the setting of hierarchically hyperbolic groups, median-quasiconvexity is equivalent to \emph{hierarchical quasiconvexity} (see \cite[\S5]{behrstockhagensisto:hierarchically:2}) by work of Russell--Spriano--Tran, who also show that the class of hierarchically quasiconvex subsets includes the other classes mentioned above \cite[Prop.~5.11, Thm~6.3]{russellsprianotran:convexity}. Other examples include multicurve stabilisers in mapping class groups and graphical subgroups of graph products of HHGs.

\begin{proposition} \label{prop:convexity}
Let $X$ be a coarse median space and let $Q$ be a finite-dimensional CAT(0) cube complex. Suppose that there exists a quasimedian coarse surjection $f:X\to Q$. For any number $k$ there is a number $r$ such that if $Y\subset X$ is $k$--median-quasiconvex, then there is a convex subcomplex $Y'\subset Q$ with $\dist_{Haus}(f(Y),Y')\le r$.

Conversely, if $\bar f:Q\to X$ is a quasimedian coarse surjection, then there is a constant $k_0$ such that $\bar f(Y')$ is $k_0$--median-quasiconvex for any convex subcomplex $Y'$ of $Q$.
\end{proposition}

\begin{proof}
We may assume that $f$ maps $Y$ into the $0$--skeleton of $Q$. According to \cite[Lem.~2.18]{haettelhodapetyt:coarse}, the convex hull $Y'=\hull f(Y)$ can be obtained from $f(Y)$ by taking medians at most $\dim Q$ times. More precisely, if we set $Y_0=f(Y)$ and define $Y_{i+1}=\mu_Q(Q^0,Y_i,Y_i)$ for $i\ge0$, then $Y_{\dim Q}=Y'$. Using this, the first statement easily follows from the quasimedian property and coarse surjectivity of $f$. The converse statement is obvious.
\end{proof}

The quasimedian property of $\Phi$ established in Theorem~\ref{thm:colourable_HHG_QC_body} lets us apply Proposition~\ref{prop:convexity} to colourable HHGs.

\begin{corollary} \label{cor:convexity_correspondence}
Let $G$ be a colourable HHG. The map $\Phi$ gives a correspondence, in the sense of Proposition~\ref{prop:convexity}, between median-quasiconvex (alias hierarchically quasiconvex) subsets of $G$ and convex subcomplexes of $Q$.
\end{corollary}

In view of this, one can recover the colourable case of \cite[Thm~3.5]{haettelhodapetyt:coarse}, and its consequences for bounded packing, from the Helly property for convex subcomplexes of CAT(0) cube complexes \cite[Thm~2.2]{roller:poc}. In the more general setting it is deduced from work of Chepoi--Dragan--Vax\`es \cite[Thm~5.1]{chepoidraganvaxes:core} on the coarse Helly property for quasiconvex subsets of hyperbolic graphs.

\begin{corollary} \label{cor:helly}
Let $\cal Y$ be a collection of $k$--median-quasiconvex subsets of a colourable HHG $(G,\s)$ that either is finite or contains a bounded element. For any number $r$ there is a corresponding number $R$ such that if the subsets $Y\in\cal Y$ are pairwise $r$--close, then there is some $g\in G$ that is $R$--close to every $Y\in\cal Y$.
\end{corollary}

In the setting of colourable HHGs, this generalises \cite[Prop~6.3]{antolinmjsistotaylor:intersection}, the proof of which relies in an essential way on the fact that stable subgroups are hyperbolic. An important point about Corollary~\ref{cor:helly} is that the constant $R$ is independent of the cardinality of the collection $\cal Y$.

In fact, for hyperbolic spaces of finite asymptotic dimension, the Chepoi--Dragan--Vax\`es result \cite[Thm~5.1]{chepoidraganvaxes:core} can itself be recovered from Theorem~\ref{thm:tfae_hyp_asdim_quasicubical} and Proposition~\ref{prop:convexity} in exactly the same way.

\subsubsection*{\textbf{\emph{Recovering \cite[Thm~1.4]{durhamminskysisto:stable} for groups.}}}  ~

We conclude with a brief description of how to recover a result of Durham--Minsky--Sisto about approximating finite subsets of colourable HHGs. Though the proof is simpler, the version obtained here gives less control on the dimension, and the equivariant version stated as \cite[Thm~4.1]{durhamminskysisto:stable} does not directly follow from the results of this article. 

Let $(G,\s)$ be a colourable HHG, and let $\psi:G\to Q$ be a quasimedian quasiisometry to a finite-dimensional CAT(0) cube complex, as provided by Theorem~\ref{thm:colourable_HHG_QC_body}, with quasiinverse $\bar\psi:Q\to G$. Given a finite subset $F\subset G$, let $Q_F=\hull\psi(F)$. Note that $Q_F$ is finite. By Proposition~\ref{prop:convexity} and \cite[Prop.~5.11]{russellsprianotran:convexity}, the map $\phi_F=\bar\psi|_{Q_F}:Q_F\to H_\theta(F)$ is a $K$--quasimedian $(K,K)$--quasiisometry, where $H_\theta(F)$ is the $\theta$--hull from \cite[Def.~6.1]{behrstockhagensisto:hierarchically:2}, and $K$ is independent of $F$.

\begin{corollary}[{\cite[Thm~1.4]{durhamminskysisto:stable}}]
Suppose that $F_1,F_2\subset G$ have $|F_1|=|F_2|=k$ and $\dist_{Haus}(F_1,F_2)\le1$. There is a constant $K'=K'(k,K)$ and a CAT(0) cube complex $Q'$ with a $K$--quasimedian $(K,K)$--quasiisometric embedding $\phi':Q'\to G$ such that $\dist(\phi'\eta_i,\phi_i)\le K'$, where $\eta_i:Q_{F_i}\to Q'$ are hyperplane deletion maps that delete at most $2kK$ hyperplanes.
\end{corollary}

\begin{proof}
Because $\dist_{Haus}(\psi(F_1),\psi(F_2))\le2K$, the convex hulls $Q_{F_1}$ and $Q_{F_2}$ only differ by at most $2kK$ hyperplanes. The cubical gate map $Q_{F_1}\to Q_{F_2}$ restricts to an isomorphism of subcomplexes $Q'_1\to Q'_2$. As an abstract cube complex, $Q'_i$ is dual to the hyperplanes crossing both $Q_{F_1}$ and $Q_{F_2}$, so we have $\eta_i$ as desired. Observing that $\dist_{Haus}(Q'_1,Q'_2)\le2K$ and letting $\phi'=\phi_{F_1}|_{Q'_1}$, it is straightforward to check the coarse agreement of $\phi'\eta_i$ and $\phi_i$.
\end{proof}

\bibliographystyle{alpha}
\bibliography{bibtex}
\end{document}